\documentclass[preprint, 11pt]{amsart}
\usepackage{amsfonts,amssymb,amsmath,amsthm,amscd,mathtools,multicol, vwcol,tikz, tikz-cd,caption,enumerate,mathrsfs,geometry,tgpagella,stmaryrd}
\usetikzlibrary{arrows,chains,matrix,positioning,scopes}
\usepackage{inputenc}
\usepackage[foot]{amsaddr}
\usepackage[pagebackref=true, colorlinks, linkcolor=blue, citecolor=red]{hyperref}
\usepackage{fullpage}
\usepackage[all]{xy}
\newtheorem{theorem}{Theorem}[section]
\newtheorem*{theorem*}{Theorem}
\theoremstyle{corollary}
\newtheorem{corollary}{Corollary}[theorem]
\newtheorem*{corollary*}{Corollary}
\theoremstyle{remark}
\newtheorem{remark}[theorem]{Remark}
\newtheorem*{remarks*}{Remarks}
\newtheorem{example}{\bf Example}[section]

\theoremstyle{proposition}
\newtheorem{proposition}[theorem]{Proposition}

\numberwithin{equation}{subsection}
\makeindex
\begin{document}
	
	\title{Differential Galois groups of differential central simple algebras and their projective representations}
	\author{Manujith K. Michel \and Varadharaj R. Srinivasan}
	
	\maketitle
\begin{abstract}
	Let $F$ be a $\delta-$field (differential field) of characteristic zero with an  algebraically closed field of constants $F^\delta$,  $A$ be a $\delta-F-$central simple algebra, $K$ be a Picard-Vessiot extension for the $\delta-F-$module $A$ and $\mathscr G(K|F)$ be the $\delta-$Galois group of $K$ over $F.$  We prove that a $\delta-$field extension $L$ of $F,$ having $F^\delta$ as its field of constants, splits the $\delta-F-$central simple algebra $A$ if and only if the $\delta-$field $K$ embeds in $L.$ 
	We then extend the theory of  $\delta-F-$matrix algebras over a $\delta-$field $F,$ put forward by Magid \& Juan (\cite{JM08}), to arbitrary $\delta-F-$central simple algebras.  In particular,  we  establish a natural bijective correspondence between the isomorphism classes of $\delta-F-$central simple algebras of dimension $n^2$ over $F$ that are split  by the $\delta-$field $K$ and the classes of  inequivalent representations of the algebraic group $\mathscr G(K|F)$ in $\mathrm{PGL}_n(F^\delta).$ We show that  $\mathscr G(K|F)$ is a reductive or a solvable algebraic group if and only if $A$ has certain kinds of $\delta-$right ideals.   \end{abstract}

\section{Introduction} All rings considered in this paper are unital and of characteristic zero.      A \emph{differential ring} or a \emph{$\delta-$ring} $A$ is a pair $(A, \delta_A),$ where $A$ is  a ring and  $\delta_A: A\to A$ is an additive map such that $\delta_A(xy)=\delta_A(x)y+x\delta_A(y)$ for all $x, y \in A.$ A \emph{$\delta-$right ideal} $I$ of a $\delta-$ring $A$ is a  right ideal of $A$ such that $\delta_A(x)\in I$ for all $x\in I.$ A \emph{$\delta-$ring homomorphism} $\phi: (A,\delta_A)\to (B, \delta_B)$ is a unital ring homomorphism such that $\phi(\delta_A(a))=\delta_B(\phi(a))$ for all $a\in A.$ A $\delta-$ring $B$ is said to be a \emph{$\delta-$ring extension} of a $\delta-$ring $A$ if  there is an injective $\delta-$ring homomorphism from $A$ to $B.$ Let $A$ be a $\delta-$ring. By a \emph{differential $A-$module} or a \emph{$\delta-A-$module} $M,$  we mean a pair $(M, \delta_M),$ where $M$ is a  finitely generated unital $A-$module  with an additive map $\delta_M: M\to M$ such that  $\delta_M(ax)=\delta_A(a)x+a\delta_M(x),$ for all $a\in  A$ and $x\in M.$ For a $\delta-A-$module $M,$ we denote $\mathrm{ker}(\delta_M)$ by $M^{\delta}.$ Note that $A$  is a $\delta-A-$module, $A^\delta$ is a subring of $A$, $A^\delta$ is a field whenever $A$ is a field and that  $M^\delta$ is an $A^\delta-$module.    

Let $F$ be a $\delta-$field.  A $\delta-F-$module $M$ is said to be  \emph{trivial}   if $\mathrm{dim}_F(M)= \mathrm{dim}_{F^{\delta}}(M^\delta).$  
An $F-$algebra $A$ is said to be a \emph{$\delta-F-$ algebra} if $A$ is a $\delta-$ring  extension of $F.$ Note that a $\delta-F-$algebra $A$ with $\mathrm{dim}_F(A)<\infty$ is also a $\delta-F-$module. A $\delta-F-$ homomorphism between $\delta-F-$ algebras is an $F-$algebra homomorphism which is also a $\delta-$homomorphism of rings. A \emph{differential $F-$central simple algebra} or a  $\delta-F-$central simple algebra  $A$ is a $\delta-F-$algebra $A$ which is also a central simple $F-$algebra.  The central simple algebra $\mathrm{M}_n(F)$  has the natural derivation $\delta^c,$ called the \emph{coordinate-wise derivation}, defined by $\delta^c(a_{ij}):=(\delta_F(a_{ij}))$ for $(a_{ij})\in \mathrm{M}_n(F).$ Furthermore, the $\delta-$ring $(\mathrm{M}_n(F), \delta^c)$ is a  $\delta-F-$central simple algebra.

Let $A$ be a $\delta-F-$ central simple algebra and $K$ be a $\delta-$field extension of $F.$  Then,  the $\delta-$ring $(A\otimes_F K, \delta_A\otimes \delta_K),$ where the derivation $\delta_A\otimes \delta_K$  maps $a\otimes k$ to  $\delta_A(a)\otimes  k+ a\otimes \delta_K(k),$ is  a $\delta-K-$central simple algebra.  The $\delta-$field $K$ is said to \emph{split} the $\delta-F-$ central simple algebra $A$  if there is a  $\delta-K-$algebra isomorphism between   $\left(A\otimes_F K, \delta_A\otimes \delta_K \right)$ and  $\left(\mathrm{M}_n(K), \delta^c\right).$  

Before we proceed to state the main results of this paper, we shall state a few more definitions and facts from Picard-Vessiot theory. A $\delta-$field extension $K$ of $F$ is said to be a \emph{Picard-Vessiot extension}\footnote{If $F^{\delta}$  is an algebraically closed field, for any given $\delta-F-$ module,  Picard-Vessiot extensions exist and are unique up to $\delta-F-$isomorphisms.} for a $\delta-F-$module $M$ if \begin{enumerate}\item $K^{\delta}=\{c\in K\ |\ \delta_K(c)=0\}=F^{\delta}.$ \item  For any $F-$basis $e_1,\dots, e_n$ of $M$ and a matrix $(a_{ij})\in \mathrm{M}_n(F)$ such that $\delta_M(e_i)=-\sum^n_{j=1}a_{ji}e_j,$ there is a matrix $(k_{ij})\in\mathrm{GL}_n(K)$ such that $\delta^c(k_{ij})=(a_{ij})(k_{ij}).$ \item $K$ is generated as a field over $F$ by the entries $k_{ij}.$\end{enumerate}

If $\omega=\mathrm{det}(k_{ij})$ then the $F-$subalgebra $\mathscr R\subseteq K$ generated by $\{k_{ij}\ | \ 1\leq i, j\leq n\}$ and $\omega^{-1}$ is called the \emph{Picard-Vessiot ring}  for the $\delta-F-$module $A.$ The field of fractions of $\mathscr R$ is $K$ and $\mathscr R$ is integrally closed (in $K$).
Let  $K$ be a Picard-Vessiot extension for the  $\delta-F-$module $M.$
The group of all  $\delta-$automorphisms of $K$ that fixes the elements of $F$ is called the \emph{differential Galois group} or the \emph{$\delta-$Galois group} of $K$ over $F$ or equivalently, the $\delta-$Galois group  of the $\delta-F-$module $M$ and it is denoted by $\mathscr{G}(K|F).$

 For any $\sigma\in \mathscr{G}(K|F),$ since $\sigma(a_{ij})=a_{ij},$ it is easily seen that $$(k_{ij})^{-1}(\sigma(k_{ij}))=:(c_{ij\sigma})\in\mathrm{GL}_n(F^{\delta})$$ and that $\phi: \mathscr{G}(K|F)\to\mathrm{GL}_n(F^{\delta});$ $\phi(\sigma)= (c_{ij\sigma})$ is a faithful group homomorphism. In fact, the image of $\phi$ is a Zariski closed subgroup of $\mathrm{GL}_n(F^{\delta})$ and  therefore the $\delta-$Galois group can be naturally  identified with the $F^{\delta}-$points of a linear algebraic group defined over $F^{\delta}.$ The group $\mathscr G(K|F)$ stabilizes the  Picard-Vessiot ring $\mathscr R$ and $\mathrm{maxspec}(\mathscr R)$ is a torsor of $\mathscr G(K|F)$ over $F.$

We now state the main results of this paper
%Our theory establishes a natural bijective correspondence between the collection
\begin{theorem}\label{Mainresult} Let $F$ be a $\delta-$field, $F^\delta$ be algebraically closed  and $A$ be a $\delta-F-$central simple algebra. Let $L$ be a $\delta-$field extension of $F$ with $F^{\delta}=L^{\delta}$ and $K$ be a Picard-Vessiot extension for the $\delta-F-$module $A.$ 
	
	\begin{enumerate}[(i)] \item  \label{split-trivial-equivalence}   The $\delta-$field $L$ splits the $\delta-F-$algebra $A$ if and only if $(A\otimes_F L, \delta_A\otimes \delta_L)$ is a trivial $\delta-L-$module.
		\item  \label{existence-minimal} The $\delta-$field  $L$ splits the $\delta-F-$algebra $A$ if and only if the Picard-Vessiot extension $K$ embeds  into the $\delta-$field $L.$
		 
		\item \label{reductivecase} The $\delta-$Galois group $\mathscr{G}(K|F)$ is  reductive  if and only if $A=\oplus^l_{i=1}\mathfrak{a}_i,$ where each $\mathfrak{a}_i$ is minimal among all $\delta-$right ideals of $A.$ 
		
		\item \label{liouvilliancase} Let $F^0$ be the algebraic closure of $F$ in $K.$ Then, the identity component of $\mathscr{G}(K|F)$ is  solvable if and only if $A\otimes_F F^0$ is a split $F^0-$algebra, that is, $A\otimes_F F^0\cong \mathrm{M}_n(F^0)$ as $F^0-$algebras and that there is a chain of $\delta-$right ideals $$\mathfrak a_1\subsetneq \mathfrak a_2\subsetneq\cdots\subsetneq \mathfrak a_n=A\otimes_F F^0$$ of $A\otimes_F F^0$ such that $\mathrm{dim}_{F^0}(\mathfrak a_j)=jn.$ 
	\item \label{maxstable} If $A$ has a  maximal subfield that is stabilized by $\delta_A$  then  the identity component of $\mathscr{G}(K|F)$ is a torus.
	
\item \label{trivialtorsor}If the Picard-Vessiot ring $\mathscr R\subseteq K$ is a trivial torsor of $\mathscr G(K|F)$ over $F$ then $A$ must be a split $F-$algebra.
\end{enumerate}
		\end{theorem}
	Now we shall explain the significance  of our results. For a $\delta-F-$central simple algebra $A,$ the existence of Picard-Vessiot splitting fields was first proved in \cite[Theorem 1]{JM08}. Their approach can be explained as follows. First tensor up $A$ by a finite Galois extension $L$ of $F$ so that $A\otimes_F L$ becomes a split $L-$algebra. Second, choose  an $L-$isomorphism $\psi: A\otimes_F L\to \mathrm{M}_n(L)$ of $L-$algebras and observe that  $\psi$ can be made into a $\delta-L-$isomorphism between $(A\otimes_F L, \delta_A\otimes \delta_L)$ and $(\mathrm{M}_n(L), D),$ where $\delta_L$  is the unique derivation  on $L$ that restricts to the derivation $\delta$ on $F$ and $D:=\psi\circ(\delta_A\otimes \delta_L)\circ \psi^{-1}.$  Third,  observe that there is a matrix $(p_{ij})\in \mathrm{M}_n(L)$ such that $(D-\delta^c) (x_{ij})=(x_{ij})(p_{ij})-(p_{ij})(x_{ij})$ (see \cite[Proposition 1]{JM08}). Fourth, show that  if $M$ is an $n-$dimensional $\delta-L-$module with a basis $e_1,\dots, e_n$ and  $\delta_M(e_i)=-\sum^n_{j=1}p_{ji}e_j$ then 
	the Picard-Vessiot extension $E$ of $L$ for the $\delta-L-$module $M$  splits the $\delta-L-$algebra $(\mathrm{M}_n(L), D).$ Finally, using the fact that a Picard-Vessiot extension $E$ of a finite Galois extension of $F$ can be embedded in some Picard-Vessiot extension $\tilde{E}$ of $F,$ it is shown that $\tilde{E}$ splits the $\delta-F-$central simple algebra $A.$  Though this method produces a  Picard-Vessiot splitting field for the $\delta-F-$central simple algebra $A,$ it does not establish any clear relationship with the $\delta-F-$algebra $A.$ Whereas, our Theorem \ref{Mainresult} (\ref{existence-minimal}) shows that the Picard-Vessiot extension for the $\delta-F-$module $A$ not only splits the $\delta-F-$algebra $A$ but also embeds in  every $\delta-$field,  including $\tilde{E},$ that splits $A$ and has $F^{\delta}$ as its field of constants.
	
	A linear algebraic group defined over an algebraically closed field of characteristic zero is reductive if and only if it is linearly reductive. Thus, if $A^{op}[\partial]$ denotes the ring of differential operators over the opposite central simple $F-$algebra $A^{op}$ then Theorem \ref{Mainresult}(\ref{reductivecase}) says that the algebraic group $\mathscr{G}(K|F)$ has a faithful and completely reducible module if and only if $A,$ as a left $A^{op}[\partial]-$module under the action $$\left(\sum^m_{j=1}a_j\partial^j\right)x=\sum^m_{j=1} \delta^j(x) a_j,\qquad a_j, x\in A,$$ is completely reducible (semisimple). When the identity component of the $\delta-$Galois group  is solvable, Theorem \ref{Mainresult}(\ref{liouvilliancase}) says that the algebraic closure $F^0$ of $F$ in the Picard-Vessiot extension for the $\delta-F-$module $A$ splits the algebra $A$ and that $A\otimes_F F^0$ has a flag of $\delta-$right ideals of length $\sqrt{\mathrm{dim}_F(A)}.$ 
	
	We shall now briefly outline the technique used in the proofs of Theorem \ref{Mainresult}(\ref{reductivecase}) and (\ref{liouvilliancase}).  Let $K$ be the Picard-Vessiot extension of the $\delta-F-$module $A.$ The Tannakian formalism in differential Galois theory states that  the Tannakian category  $\{\{A\}\}$ generated by the  $\delta-F-$module $A$ is equivalent to the category $\mathrm{Rep}_{\mathscr{G}}$ of all finite dimensional $F^\delta-$linear representations of the $\delta-$Galois group $\mathscr G:=\mathscr{G}(K|F)$ of the $\delta-F-$module $A.$ The equivalence is given by the following fiber functor over $F^{\delta}$ $$\mathscr S: \{\{A\}\}\to \mathrm{Rep}_{\mathscr G}; \quad \mathscr S(N)=(N\otimes_F K)^{\delta_N\otimes\delta_K}, \ N\in \{\{A\}\}$$  and the $\mathscr G-$action is given by $$\rho_N: \mathscr G\to \mathrm{GL}(\mathscr S(N)), \quad \rho_N(\sigma)(a\otimes k)=a\otimes \sigma(k), \quad a\in N,\  k\in K, \ \sigma\in \mathscr G.$$    Since  $A$ is an $F-$algebra, for each $\sigma\in \mathscr G,$ $\rho_A(\sigma)$ becomes an automorphism of the $F^\delta-$algebra   $\mathscr S(A).$  Thus image of $\rho_A$ lies inside $\mathrm{Aut}(\mathscr S(A)):$ $$\mathscr G\to \mathrm{Aut}(\mathscr S(A))\hookrightarrow \mathrm{GL}(\mathscr S(A)).$$
It follows from Theorem \ref{Mainresult}(\ref{existence-minimal}) that  $\mathscr S(A)\cong \mathrm{M}_n(F^\delta)$ and thus $\mathrm{M}_n(F^\delta)$ can be viewed as a $\mathscr G-$algebra where elements of $\mathscr G$ acts as automorphisms on $\mathrm{M}_n(F^\delta).$ We study this action in detail and obtain that the $\mathscr G-$stable ideals of $\mathrm{M}_n(F^\delta)$ are in an inclusion preserving bijective correspondence with the $\delta-$right ideals of $A.$ When $\mathscr G$ is reductive, we show that  $$\mathrm{M}_n(F^\delta)=\oplus^l_{i=1} I_i,$$ where each $I_i$ is minimal among all $\mathscr G-$stable ideals of $\mathrm{M}_n(F^\delta).$ Through the bijective correspondence, it is then shown that $A=\oplus^l_{i=1}\mathfrak{a}_i,$ where each $\mathfrak a_i$ is  minimal among all $\delta-$right ideals of $A.$ The converse of Theorem \ref{Mainresult}(\ref{reductivecase}) and the proof of Theorem \ref{Mainresult}(\ref{liouvilliancase}) is also handled similarly. 

From the isomorphism $\mathscr S(A)\cong \mathrm{M}_n(F^\delta),$ we also obtain a projective representation $$\mathscr G\to \mathrm{Aut}(\mathscr S(A))\cong\mathrm{Aut}(\mathrm{M}_n(F^\delta))\cong \mathrm{PGL}_n(F^\delta).$$ In the context of matrix $\delta-F-$algebras, the above projective representation coincides with the one given in \cite[p.1914]{JM08}. We also obtain the following extension of \cite[Theorem 2]{JM08}.

\begin{theorem} \label{proj-equivalence}($\mathrm{cf.}$ \cite[Theorem 4.11]{tsui-wang}, \cite[Theorem 2]{JM08}) Let $F$ be a $\delta-$field with  $F^{\delta}$ an algebraically closed field.  Let $K$ be a Picard-Vessiot extension of some $\delta-F-$module and $\mathscr G$ be its $\delta-$Galois group. Then, there is  a natural bijective correspondence $\Gamma$, induced by the Tannakian correspondence,  between $\delta-\mathrm{CSA}_{K/F, \ n},$ the collection of all $\delta-F-$algebra isomorphism classes of  central simple algebras of dimension $n^2$ that are split  by the $\delta-$field $K$ and $\mathrm{PRep}_{\mathscr G, \ n},$ the collection of all inequivalent (projective) representations of $\mathscr G(K|F)$ in $\mathrm{PGL}_n(F^\delta).$ 
\end{theorem}

We shall now explain how the above theorem can also be obtained from \cite[Theorem 4.11]{tsui-wang}. 
Let $A$ be a $\delta-F-$central simple algebra, $K$ be the Picard-Vessiot extension for the $\delta-F-$module $A$ and $[A]$ denote the equivalence class of all $\delta-F-$central simple algebras that are $\delta-F-$isomorphic to $A.$ As in \cite[Definition 3.11]{tsui-wang}, let
$\mathrm{TF}(K/F, A)$ denote the collection of all $[B],$ where  $B$ is a $\delta-F-$central simple algebra and $(B\otimes K, \delta_B\otimes \delta_K)$ and  $(A\otimes K, \delta_A\otimes \delta_K)$ are isomorphic as $\delta-K-$algebras. Then, the definitions of $\mathrm{TF}(K/F, A)$ and $\delta-\mathrm{CSA}_{K|F, \ n}$ are in fact one and the same.  Now,  by \cite[Theorem 4.11]{tsui-wang}, there is a bijective correspondence between $\mathrm{TF}(K/F, A)$ and a cohomology set $\mathrm{H}^1(\mathscr G(K|F), \mathrm{PGL}_n(F^\delta))$  consisting of cohomologous one cocyles that are maps of algebraic varieties. It can be seen that the action of the group $\mathscr G$ on  $\mathrm{PGL}_n(F^\delta)$ defined in  \cite[Remark 4.4]{tsui-wang} is the trivial action. Therefore, $\mathrm{H}^1(\mathscr G(K|F), \mathrm{PGL}_n(F^\delta))$ is the same as $\mathrm{PRep}_{\mathscr G, \ n}.$

If $A$ has a maximal subfield $L,$ then from a theorem of Hochschild  (\cite[Theorem 6]{Hoc55}), there is a derivation $\delta_A$ on $A$ stabilizing both  the fields $L$ and $F.$    For such a derivation $\delta_A,$ with $F^{\delta}$ an algebraically closed field,  Theorem \ref{Mainresult}(\ref{maxstable}) says that the identity component of the $\delta-$Galois group of the $\delta-F-$module $A$ is a torus. 

The groups $\mathrm{GL}_n,$ $\mathrm{SL}_n$ or any  connected solvable group have only  trivial torsors. Therefore, by Theorem \ref{Mainresult}(\ref{trivialtorsor}), they cannot appear as the $\delta-$Galois group of the $\delta-F-$module $(A, \delta_A)$  unless $A$ is a split $F-$algebra.

\subsection{Organisation of the paper} In Section \ref{pvtheory}, we provide a detailed account of Picard-Vessiot theory for differential modules following the book \cite{MvdP03}.  In Section \ref{DCSA}, we prove Theorem \ref{Mainresult} (\ref{split-trivial-equivalence}) and (\ref{existence-minimal}) and examine torsion elements of the $\delta-F-$Brauer monoid. We show in Section \ref{projrep} that the underlying $\delta-F-$module structure of a $\delta-F-$central simple algebra $A$ naturally yields a  projective representation of the $\delta-$Galois group of $A.$ We then use this representation to prove Theorem \ref{proj-equivalence}. The proofs of Theorem \ref{Mainresult}(\ref{reductivecase}) and (\ref{liouvilliancase}) can be found in Section \ref{reducible-solvable} and the proofs of Theorem \ref{Mainresult}(\ref{maxstable}) and (\ref{trivialtorsor}) can be found in Section \ref{maxsubfield-stability}.

	\section{Picard-Vessiot Theory for Differential Modules.}\label{pvtheory}

Let $M$ be a $\delta-F-$module.  Given an $F-$basis $e_1,\dots, e_n$ of $M,$ there is a matrix $(a_{ij})\in \mathrm{M}_n(F)$ such that $$\delta_M(e_i)=-\sum^n_{j=1}a_{ji} e_j$$ and the differential equation $\delta^c(Y)=(a_{ij})Y$ is called the \emph{matrix differential equation associated to the $\delta-F-$module $M$} (with respect to the basis $e_1,\dots, e_n$ of $M$).   A matrix $(k_{ij})\in\mathrm{GL}_n(F)$ is said to be a \emph{fundamental matrix} for the differential equation $\delta^c(Y)=(a_{ij})Y$ if $\delta^c\left(k_{ij}\right)=(a_{ij})(k_{ij}).$ 
	
	\begin{proposition} \label{gaugeequivalence}(\cite[Section 1.2, p.7]{MvdP03})
		Let  $e_1,\dots, e_n$ and $\tilde{e}_1,\dots,\tilde{e}_n$ be two $F-$basis of $M$ and $(a_{ij}), (b_{ij})\in \mathrm{M}_n(F)$ be matrices  such that  $$\delta_M(e_i)=-\sum^n_{j=1}a_{ji} e_j,\quad\qquad \delta_M(\tilde{e}_i)=-\sum^n_{j=1}b_{ji} \tilde{e}_j.$$ Then, the matrix $(b_{ij})$  is \emph{gauge equivalent} to $(a_{ij})$ over $F$ by a matrix $(m_{ij})\in\mathrm{GL}_n(F).$ That is, $$(a_{ij})=(m_{ij})^{-1}(b_{ij})(m_{ij})-(m_{ij})^{-1}\delta^c(m_{ij}); \quad \text{for some matrix} \ (m_{ij})\in\mathrm{GL}_n(F).$$  In particular, any two matrix differential equation associated to the $\delta-F-$module $M$ are gauge equivalent over $F.$ Furthermore,  $(k_{ij})\in \mathrm{GL}_n(F)$ is a fundamental matrix for $\delta^c(Y)=(b_{ij})Y$ if and only if $(m_{ij})(k_{ij})\in \mathrm{GL}_n(F)$ is a fundamental matrix for $\delta^c(Y)=(a_{ij})Y.$
	\end{proposition}

From the above proposition, it follows that a $\delta-F-$module can be associated to a unique (up to gauge equivalence) differential equation and thus making the definition of a Picard-Vessiot extension for a  $\delta-F-$module, given in the introduction, unambiguous.

	%%%%%%%%%%%%%%%%%%%%%%%%%%%%%%%%%%%%%%%%%%%%%%%%%%	%%%%%%%%%%%%%%%%%%%%%%%%%%%%%%%%%%%%%%%%%%%%%%%%%%

\subsection{Construction of a Picard-Vessiot Extension} If  $F^{\delta}$ is an algebraically closed field then one can construct  a Picard-Vessiot extension for a $\delta-F-$module  $M.$ We shall briefly outline the construction and record certain facts from Picard-Vessiot theory for later usage   (\cite[Lemma 1.17 \& Proposition 1.20]{MvdP03}). Let $e_1, \dots, e_n$ be an $F-$basis of $M$ and $(a_{ij})\in \mathrm{M}_n(F)$ is  a matrix such that $$\delta_M(e_i)=-\sum^n_{j=1}a_{ji}e_j.$$  Consider the polynomial ring $F[x_{ij}\ |\ 1\leq i,j\leq n]$ with the derivation that maps $$(x_{ij})\mapsto (a_{ij})(x_{ij}).$$ Let $w:=\mathrm{det}(x_{ij})$ and $\mathscr M$ be any $\delta-$ideal that is maximal among all $\delta-$ideals that do not contain $w.$ Such a $\delta-$ideal $\mathscr M$ can be shown to be a prime ideal.  The  quotient ring $$\mathscr R:= \frac{F[x_{ij}\ |\ 1\leq i,j\leq n][w^{-1}]}{\mathscr M},$$   called the \emph{Picard-Vessiot ring} for the $\delta-F-$module $M,$ is an integral domain as well as a $\delta-F-$algebra. If we denote the image of $x_{ij}$ in $\mathscr R$ by $k_{ij}$ and $w$ by $\omega$, then the matrix $(k_{ij})$ belongs to $\mathrm{GL}_n(\mathscr R)$ and $\delta^c(k_{ij})=(a_{ij})(k_{ij}).$ Therefore, as an $F-$algebra, $\mathscr R$ is finitely generated by the entries of the matrix $(k_{ij})$ and $\omega^{-1}.$  Let $\delta_K$ be the extension of the derivation on $\mathscr R$ to the field of fractions $K$ of $\mathscr R.$  When $F^{\delta}$ is an algebraically closed field, $K^{\delta}=F^{\delta}$ and it follows that $K$  is a Picard-Vessiot extension for the $\delta-F-$module $M$ as $K$ is generated as a field by the entries of the fundamental matrix $(k_{ij})$ and $F.$ 

Note that the construction of $K$ is independent of the  choice of basis of $M$ as $\mathrm{GL}_n(K)$ will have fundamental matrices for any other differential equation $\delta^c(Y)=(b_{ij}) Y,$  where $(b_{ij})\in \mathrm{M}_n(F)$  is gauge equivalent  to the matrix $(a_{ij})$ over $F.$  

Recall that the $\delta-$Galois group $\mathscr G:=\mathscr{G}(K|F)$ of the $\delta-F-$module $M$  is the group of all $\delta-F-$automorphisms  of $K.$ For any $\sigma\in \mathscr G,$ it is easily seen that $(k_{ij})^{-1}(\sigma(k_{ij}))=(c_{ij\sigma})\in\mathrm{GL}_n(F^{\delta})$ and that the mapping $\mathscr G\mapsto\mathrm{GL}_n(F^{\delta})$ is a faithful group homomorphism.  Every element $(c_{ij})\in\mathrm{GL}_n(F^{\delta})$ acts on the coordinate ring $F^{\delta}[\mathrm{GL}_n]=F^{\delta}[x_{ij}\ |\ 1\leq i,j\leq n][w^{-1}]$ by right translation: $(x_{ij})\mapsto (x_{ij})(c_{ij}).$  This action extends to an action on $$F\otimes_{F^{\delta}} F^{\delta}[\mathrm{GL}_n]\cong F[x_{ij}\ |\ 1\leq i,j\leq n][w^{-1}]$$ by $\mathrm{GL}_n(F^\delta)$ acting trivially on $F.$  Under this action, the $\delta-$Galois group $\mathscr G$ can be identified with those matrices $(c_{ij})\in\mathrm{GL}_n(F^{\delta})$ that stabilizes the maximal $\delta-$ideal $\mathscr M.$ Thus, the identification \begin{equation*}\mathscr G\mapsto\mathrm{GL}_n(F^{\delta});\quad \sigma\to (c_{ij\sigma})\end{equation*} 
makes $\mathscr G$ a Zariski closed subgroup of $\mathrm{GL}_n(F^{\delta})$ (\cite[Observation 1.26 \& Theorem 1.27]{MvdP03}). 

Similar to the Galois theory of polynomial equations,  there is a Galois correspondence between the  algebraic subgroups $\mathscr{H}$ of $\mathscr G$ and the $\delta-$subfields $L$ that are intermediate to $K$ and $F:$ 
\begin{align*} &\mathscr H \mapsto K^\mathscr{H}=\{\alpha\in K\ | \ \sigma(\alpha)=\alpha\ \text{for all}\ \sigma\in \mathscr{H}\} \\  & \mathscr G(K|L)=\{\sigma\in \mathscr G(K|F)\ | \ \sigma(l)=l \ \text{for all}\ l\in L\}\mapsfrom L.\end{align*} In particular, $F=E^\mathscr{G}$ and if $\mathscr G^0$ is the identity component of $\mathscr G$ then $F^0:=E^{\mathscr G^0}$ is the algebraic closure of $F$ in $K.$

%%%%%%%%%%%%%%%%%%%%%%%%%%%%%%%%%%%%%

%%%%%%%%%%%%%%%%%%%%%%%%%%%%%%%%%%%%%%%%%%%%%%%%%%%%%%%%%%%%%%%%%%%%%%%%%%%%%%%%%%%%%%%%%%%%%%%%%%%%%%%
\subsection{The Tannakian Formalism in Differential Galois Theory} \label{tannakianformalism} Henceforth, all unadorned tensor products of modules are over the field $F.$
A map $\phi: (M, \delta_M)\to (N, \delta_N)$ is said to be a $\delta-F-$module homomorphism if $\phi$ is a homomorphism of $F-$modules such that $\phi(\delta_M(x))=\delta_N(\phi(x))$ for all $x\in M.$

\begin{proposition} (\cite[Chapter II, Corollary 1]{Kol-Book}) Let $M$ be a $\delta-F-$module and \begin{equation*}\mu_M:M^\delta\otimes_{F^{\delta}} F\to M,\qquad \mu_M(x\otimes f)= fx, \quad f\in F, \ x\in M.\end{equation*}
	Then, $\mu_M$ is an injective map of $\delta-F-$modules.
\end{proposition}
It is seldom the case that  $\mu_M$ in  the above proposition is surjective and in fact, $\mu_M$ is  surjective, or equivalently, $\mu_M$ is an isomorphism if and only if $M$ is a trivial $\delta-F-$module. 

Let  $\delta^c(Y)=(a_{ij})Y$ be the matrix differential equation  associated to a $\delta-F-$module $M$ with respect to an $F-$basis $e_1,\dots, e_n$ of $M.$  Let $K$ be a $\delta-$field extension of $F$ with $K^\delta=F^\delta$ and $\lambda: M\otimes K\to K^n$ be the isomorphism of $K-$vector spaces defined by $$\lambda(\sum^n_{i=1}e_i\otimes k_i)= \begin{pmatrix}k_1\\ k_2\\ \vdots\\ k_n\end{pmatrix},\quad k_i\in K.$$ 
Then, $\lambda$ becomes an isomorphism of $\delta-K-$modules $(M\otimes K, \delta_M\otimes \delta_K)$ and $(K^n, \delta_{K^n}),$ where \begin{equation}\label{lambdadefn}\delta_{K^n} \begin{pmatrix}k_1\\ k_2\\ \vdots\\ k_n\end{pmatrix}=\delta^c \begin{pmatrix}k_1\\ k_2\\ \vdots\\ k_n\end{pmatrix}-(a_{ij})\begin{pmatrix}k_1\\ k_2\\ \vdots\\ k_n\end{pmatrix}.\end{equation}

Therefore, an element $\sum^n_{i=1}e_i\otimes k_i\in (M\otimes K)^{\delta_M\otimes \delta_K}$ if and only if the column vector $\lambda(\sum^n_{i=1}e_i\otimes k_i)\in K^n$ is a solution of $\delta^c(Y)=(a_{ij})Y.$ Thus, $\lambda$ restricts to an isomorphism of the $F^{\delta}-$ vector space $(M\otimes K)^{\delta_M\otimes \delta_K}$ and $V:=\{v\in K^n\ | \ \delta^c(v)=(a_{ij})v\},$ where $$n= \mathrm{dim}_F(M)\geq \mathrm{dim}_{F^\delta}\left((M\otimes K)^{\delta_M\otimes \delta_K}\right)= \mathrm{dim}_{F^{\delta}}(V).$$
It now follows that $(M\otimes K, \delta_M\otimes \delta_K)$ is a trivial $\delta-K-$module, that is, $(M\otimes K, \delta_M\otimes \delta_K)$ has  a $K-$basis $x_1,\dots, x_n\in (M\otimes K)^{\delta_M\otimes \delta_K}$ if and only if the matrix $(k_{ij})\in \mathrm{GL}_n(K),$ whose $j-$th column is $\lambda(x_j),$ is a fundamental matrix for $\delta^c(Y)=(a_{ij})Y;$ in which case, for $1\leq j\leq n,$ $x_j=\sum^n_{i=1} e_i\otimes k_{ij}.$   Using these observations, one can provide the following alternate definition of a Picard-Vessiot extension for a $\delta-F-$module.

A  $\delta-$field extension $K$ of $F$ is called a \emph{Picard-Vessiot extension}  for the $\delta-F-$module $M$ if it has the following properties  (\cite[Exersice 1.16]{MvdP03}).
\begin{enumerate}[(i)] \item  $K^\delta=F^{\delta}.$ 
	\item \label{muisomorphism} $(M\otimes K, \delta_M\otimes\delta_K)$ is a trivial  $\delta-K-$module.
	\item \label{generation}If $e_1,\dots, e_n$ is a basis of $M$ then $K$ is generated as a field over $F$ by the coefficients of all $v\in (M\otimes K)^{\delta_M\otimes \delta_K}$ with respect to the $K-$basis $e_1\otimes 1, \dots, e_n\otimes 1.$
\end{enumerate}

Let $(M\otimes K, \delta_M\otimes\delta_K)$ be a trivial  $\delta-K-$module, where $K$ is a Picard-Vessiot extension for some $\delta-F-$module. The coordinate-wise action of the $\delta-$Galois group $\mathscr G:= \mathscr G(K|F)$ on $K^n,$ under the isomorphism $\lambda$ defined in Equation (\ref{lambdadefn}), becomes  the $1\otimes \mathscr G-$action on $M\otimes K:$ $(1\otimes \sigma)(\sum^n_{i=1} x_i\otimes k_i)=\sum^n_{i=1} x_i\otimes \sigma(k_i).$ Since the $\mathscr G$ action on $K^n$ stabilizes  the solution space $V,$ the $\mathscr G-$action on $M\otimes K$ stabilizes  $(M\otimes K)^{\delta_M\otimes \delta_K}$ and we obtain that the group representations  \begin{equation*} \mathscr G\to \mathrm{GL}(V) \quad \text{and}\quad \mathscr G\to \mathrm{GL}((M\otimes K)^{\delta_M\otimes \delta_K})\end{equation*} 
are equivalent. For a fundamental matrix $(k_{ij})\in \mathrm{GL}_n(K)$ of the differential equation $\delta^c(Y)=(a_{ij})Y,$  $\delta^c\left((k_{ij})^{-1}(\sigma(k_{ij}))\right)=0$ and thus  $(k_{ij})^{-1}(\sigma(k_{ij}))\in \mathrm{GL}_n(F^{\delta}).$ The columns of $(k_{ij})$ form a basis of $V$ and with respect to this basis the representation $\mathscr G\to \mathrm{GL}(V)$ becomes the  representation  \begin{equation}\label{repcocycle}\mathscr G\to \mathrm{GL}_n(F^{\delta});\qquad \sigma\mapsto (k_{ij})^{-1}(\sigma(k_{ij})).\end{equation}

 Let  $M^{\vee}$ be the dual $\delta-F-$module\footnote{$\delta_{M^{\vee}}(f)(m)=f(\delta_M(m))-\delta_F(f(m)),$ for all $f\in M^{\vee},$ $m\in M.$} of the $\delta-F-$module $M$. If $(k_{ij})\in \mathrm{GL}_n(K)$ is a fundamental matrix of  the differential equation $\delta^c(Y)=(a_{ij})Y$ associated to the $\delta-F-$module $M$ then $(k_{ji})^{-1} \in \mathrm{GL}_n(K)$ is a fundamental matrix of  the differential equation $\delta^c(Y)=-(a_{ji})Y,$  which is a differential equation associated to the dual $\delta-F-$module $M^{\vee}.$ Thus,   $M$ and $M^{\vee}$ have the same Picard-Vessiot extension.
 
  Denote  the category of $\delta-F-$modules formed by taking finite direct sums of  tensor products $(M^{\otimes n_1} \otimes (M^{\vee})^{\otimes n_2}, \delta^{\otimes n_1}_M \otimes \delta^{\otimes n_2}_{M^{\vee}})$ and their subquotients by $\{\{M\}\}.$   Then $\{\{M\}\}$ is $F^{\delta}-$linear, rigid tensor category, which is also a full subcategory of the category $\mathrm{Diff}_F$ of $\delta-F-$modules.  The forgetful functor from $\{\{M\}\}$ to the category $\mathrm{Vec}_F$ of finite dimensional  $F-$vector spaces is a fiber functor.  

If $F^{\delta}$ is an algebraically closed field then a theorem of Deligne (\cite[Section 9]{Del90}), guarantees a fiber functor from $\{\{M\}\}$ to $\mathrm{Vec}_{F^{\delta}}$ and the target of this functor can be shown to be equivalent to the category $\mathrm{Rep}_\mathscr G$ of all finite dimensional representations of a linear algebraic group $\mathscr G$ defined over $F^{\delta}.$ The existence of this fiber functor is equivalent to that of a Picard-Vessiot extension for the $\delta-F-$module $M$ and the group $\mathscr G$ can be identified with the $\delta-$Galois group of the Picard-Vessiot extension.  In fact, if $K$ is a  Picard-Vessiot extension  for the $\delta-F-$module $M$ then the functor $$\mathscr S: \{\{M\}\}\to \mathrm{Vec}_{F^{\delta}};\qquad \mathscr S(N)=(N\otimes K)^{\delta_N\otimes \delta_K}, \quad N\in \{\{M\}\}$$ is a fiber functor (over $F^{\delta}$) and induces an equivalence of categories between $\{\{M\}\}$ and  $\mathrm{Rep}_\mathscr G.$  Here, the group $\mathscr G$ acts on $N\otimes K$ by $1\otimes \mathscr G$ and  stabilizes the $F^{\delta}-$module  $\mathscr S(N)=(N\otimes K)^{\delta_N\otimes \delta_K}\subset N\otimes K.$  As noted earlier,  this restriction gives a morphism of algebraic groups. 
  \begin{equation*}\label{groupaction} \rho_N: \mathscr{G}\hookrightarrow\mathrm{GL}(\mathscr S (N));\quad  \rho_N(\sigma)(x)=(1\otimes\sigma)x,
	\qquad x\in \mathscr S(N).\end{equation*} 
 One can recover a $\delta-$module, up to $\delta-F-$isomorphisms,  from a given representation of $\mathscr{G}$ as follows. Let $W\in \mathrm{Rep}_\mathscr G.$ Then there is a $\delta-F-$module $N\in \{\{M\}\}$ such that $$(N\otimes K)^{\delta_N\otimes\delta_K}=\mathscr S(N)=W$$ and that the representation $\mathscr G\to \mathrm{GL}(W)$ is equivalent to $\mathscr G\to \mathrm{GL}(\mathscr S(N)).$

As we have observed earlier, $W\otimes_{F^{\delta}} K \overset{\mu_N}{\cong} N\otimes K$ as  $\delta-K-$modules.  The $\delta-F-$isomorphism $\mu_N$ is also $\mathscr{G}-$equivariant if we endow the following $\mathscr G-$actions. On $N\otimes K,$ we have the $1\otimes \mathscr G$ action:  \begin{equation}\label{actiononsinglecoordinate}x\otimes k\mapsto (1\otimes \sigma)(x\otimes k)=x\otimes \sigma(k), \quad x\in N, \ k\in K\end{equation} and on $W\otimes_{F^{\delta}} K=\mathscr S(N)\otimes_{F^\delta} K,$ we have the $1\otimes \mathscr G\otimes_{F^\delta} \mathscr G$ action: \begin{equation}\label{actiononbothcoordinates}y\otimes k\mapsto (1\otimes \sigma)y \otimes \sigma(k), \quad y\in \mathscr S(N), \ k\in K.\end{equation}

 Since $K^{\mathscr G}=F,$ if $e_1,\dots, e_n$ is an $F-$basis of $N$ and $x\in N\otimes K,$ then $x=\sum^n_{i=1}e_i\otimes k_i$ and we have $\sigma(x)=x$ implies $\sigma(k_i)=k_i$ for each $i=1,2,\dots, n.$ Thus we obtain the following $\delta-F-$isomorphisms of $\delta-F-$modules \begin{equation*}\label{recovering module}N\cong N\otimes F\cong (N\otimes K)^\mathscr{G}\cong \left(W\otimes_{F^{\delta}} K\right)^\mathscr G.\end{equation*}

In fact, the assignment  $W\mapsto \left(W\otimes_{F^{\delta}} K\right)^\mathscr G$ is a functor from $\mathrm{Rep}_\mathscr{G}$ to $\{\{M\}\}$ that acts as an inverse for $\mathscr S$ (\cite[Remark 2.34]{MvdP03}).  We shall heavily rely upon this functor in the proofs of Theorems \ref{Mainresult} and \ref{proj-equivalence}.

%%%%%%%%%%%%%%%%%%%%%%%%%%%%%

\subsection{Completely reducible modules.} A $\delta-F-$module  $M$ is said to be \emph{completely reducible} if for every $\delta-F-$submodule $N$ of $M$ there is a $\delta-F-$submodule $N'$ of $M$ such that $M=N\oplus N'$. Note that an irreducible $\delta-F-$module is completely reducible. If we view a $ \delta-F-$module $M$ as an $F[\partial]-$module with the action $$\left(\sum^m_{j=1}\alpha_j\partial^j\right) x=\sum^m_{j=1}\alpha_j\delta^j(x),$$ for all $x\in M$ and $\alpha_j\in F$ then $M$ is  completely reducible as a $\delta-F-$module if and only if  $M$ is a completely reducible left $F[\partial]-$module. From the theory of completely reducible modules, it then follows that $M$ is  completely reducible as an $ \delta-F-$module  if and only if $M$  is an internal direct sum of a family of irreducible $\delta-$submodules of $M$ (\cite[Section 3.5, p.119]{Jac-80-book}). 

Let $M$ be a trivial $\delta-F-$module. Then, it is an internal direct sum of one dimensional trivial (irreducible) $\delta-F-$modules and therefore it is completely reducible. So, for any $\delta-F-$submodule $N,$ we have $M=N\oplus N'$ for some $\delta-F-$ submodule $N'$ of $M.$ If $\pi_N: M\to N$ is the corresponding projection onto the first coordinate, then $\pi_N$ is a $\delta-F-$epimorphism. This implies $N$ is also  a trivial $\delta-F-$submodule. More in general, any homomorphic image of trivial $\delta-F-$module is trivial. Similarly, one shows $M/N$ is also a trivial $\delta-F-$module. It is readily seen that direct sums, tensor products and dual of a trivial $\delta-F-$module is also trivial.  For a $\delta-F-$field extension $K$ of $F,$ we have the natural isomorphism of $\delta-K-$modules $$(M\otimes M^\vee\otimes K, \delta_M\otimes \delta_{M^{\vee}}\otimes \delta_K) \cong (M\otimes K\otimes _K M^\vee\otimes K, \delta_M\otimes \delta_K\otimes_K \delta_{M^{\vee}}\otimes \delta_K); a\otimes b\otimes k\mapsto a\otimes 1\otimes_K b\otimes k.$$ 
Through an induction argument, one then obtains 
 the following isomorphism of $\delta-K-$modules	$$(M^{\otimes n_1}\otimes (M^{\vee})^{\otimes n_2} \otimes K , \delta^{\otimes n_1}_M\otimes  \delta^{\otimes n_2}_{M^{\vee}}\otimes 
 \delta_K)\cong ((M\otimes K)^{\otimes n_1} \otimes (M^{\vee}\otimes K)^{\otimes n_2}, (\delta_M \otimes \delta_K)^{\otimes n_1}\otimes(\delta_{M^{\vee}}\otimes \delta_K)^{\otimes n_2}).$$  Thus, if $M\otimes K$ is a trivial $\delta-K-$module then so is  $M^{\otimes n_1}\otimes (M^{\vee})^{\otimes n_2} \otimes K$ and it follows that for any $N\in \{\{M\}\},$ $N\otimes K$ is also a trivial $\delta-K-$module.

Let $\mathscr G$ be a linear algebraic group  defined over $F^{\delta},$ $V$ be a  finite dimensional $F^{\delta}-$vector space which is also a $\mathscr{G}-$module. Then, $V$ is said to be \emph{completely reducible} if every $\mathscr{G}-$submodule $W$ of $V$ has a $\mathscr{G}-$submodule $W'\subset V$ such that $V=W\oplus W'.$

\begin{proposition}\label{completelyreducible-reductive} A $\delta-F-$module $M$ is completely reducible if and only if the $\delta-$Galois group $\mathscr G$ of $M$ is a reductive algebraic group.\end{proposition}
\begin{proof}   Let $M$ be a completely reducible $\delta-F-$module. We observe that
	a linear algebraic group $\mathscr G$ defined over a field of characteristic zero is  reductive if and only if it is linearly reductive, which in turn holds if and only if there exists a faithful completely reducible $\mathscr G-$module. Since $\mathscr S(M)$ is a faithful $\mathscr G-$module. it is enough to show that  $\mathscr S(M)$ is a competely reducible $\mathscr G-$module.

If   $W$ is a $\mathscr{G}-$submodule of $\mathscr S(M)$ then from the Tannakian equivalence of $\{\{M\}\}$ and $\mathrm{Rep}_{\mathscr{G}},$ there is a module $N_1\in \{\{M\}\}$ with $i: N_1\hookrightarrow M$ such that $\mathscr S(i):\mathscr S(N_1)=W \hookrightarrow\mathscr S(M).$ Now let $N'$ be a submodule of $M$ such that  $M=i(N_1)\oplus N'$ and observe that $\mathscr S(M)=W\oplus \mathscr S(N').$ This prove that $\mathscr S(M)$ is a completely reducible $\mathscr G-$module.   Converse is also proved in a similar way. \end{proof}
	
%%%%%%%%%%%%%%%%%%%%%	%%%%%%%%%%%%%%%%%%%%%%%%%%%%%%%%%%%%%%%%%%%%%%%%%%	%%%%%%%%%%%%%%%%%%%%%%%

%%%%%%%%%%%%%%%%%%%%%%%%%%%%%%%%%%%%%%%%%%%%%%%%%%%%%%%%%%%%%%%%%%%%%%%%%%%%%%%%%%%%%%%%%%%%%%%%%%%%%%%%%%%%%%%%%%%%%%%%%%%%%%%%%%%%%%%%%%%%%%%%%%%%%%%%%%%%%%%
\section{minimal differential splitting fields of differential central simple algebras}\label{DCSA}
Let $A$ be a central simple $F-$algebra  with a derivation map $\delta_A.$   For any $f\in F,$ we have $fx=xf$ for all $x\in A$ and therefore, applying $\delta_A$ to this equation, we obtain $\delta_A(f)x=x\delta_A(f)$ for all $x\in A$ and that $\delta_A(f)\in F.$ This implies $A$ can be viewed as a  $\delta-F-$central simple algebra.   On the other hand,  a theorem of Hochschild (\cite[Theorem]{Hoc55}) shows any derivation on $F$ can be extended to a derivation on $A$ proving that over any $\delta-$field $F,$ every central simple $F-$algebra becomes a $\delta-F-$central simple algebra.    If $F$ is endowed with the zero derivation $\theta: F\to F,$ $\theta(f)=0$ for all $f\in F,$ then for every $u\in A,$ the derivation map  $\mathrm{inn}_u: A\to A,$ given by $\mathrm{inn}_u(a)=au-ua,$  makes $(A, \mathrm{inn}_u)$ into a $\delta-F-$central simple algebra. Conversely, if $A$ is a $\delta-F-$central simple algebra and the derivation $\delta_A$ restricts to the zero derivation on $F$, then $\delta_A=\mathrm{inn}_a$ for some $a\in A$ (\cite[Theorem 8]{Jac37}). Thus, if $(A, \delta_1)$ and $(A, \delta_2)$ are $\delta-F-$central simple algebras  then since $\delta_1-\delta_2=\theta$ on $F,$ we obtain, $$\delta_1-\delta_2=\mathrm{inn}_u, \quad \text{for some }\ u\in A.$$

We recall that the matrix algebra $\mathrm{M}_n(F)$ has the coordinate-wise derivation map $$\delta^c: \mathrm{M}_n(F)\to \mathrm{M}_n(F) \ \ \text{given by}\  \delta^c(a_{ij})= (\delta_F(a_{ij}))$$ and that a $\delta-F-$central simple algebra $A$  splits over a  $\delta-$field extension $K$ of $F$ if there is a $\delta-K-$algebra isomorphism $\phi: (A\otimes K, \delta_A \otimes\delta_K)\to (\mathrm{M}_n(K), \delta^c).$ That is, $\phi$ is an isomorphism of  $K-$algebras such that $\phi\left((\delta_A\otimes \delta_K)(x) \right)=\delta^c(\phi(x))$ for all $x\in A\otimes K.$  

Let $K$ be a $\delta-$field extension of $F.$ We shall make the following simple yet important observations. The injective map of $\delta-K-$modules \begin{equation*}\label{mu-mapofalgebras}\mu_A: \left(\mathscr S(A)\otimes_{K^{\delta}} K,\  \theta\otimes_{F^{\delta}}\delta_K\right)\to (A\otimes K, \delta_A\otimes \delta_K)$$ is also a map of $\delta-K-$algebras:  $$\mu_A\left((x\otimes k_1)(y\otimes k_2)\right)=\mu_A(xy\otimes k_1k_2)=xy (1\otimes k_1)(1\otimes k_2)=x(1\otimes k_1)y(1\otimes k_2)=\mu_A(x\otimes k_1)\mu_A(x\otimes k_2).\end{equation*}
Furthermore, if $K$ is a Picard-Vessiot extension of $F$ with $\delta-$Galois group $\mathscr G$ then the $\delta-F-$algebra isomorphism $\mu_A$  is $\mathscr G-$equivariant under the actions described in Equations (\ref{actiononsinglecoordinate}) and (\ref{actiononbothcoordinates}). 
%%%%%%%%%%%%%%%%%%%%%%%%%%%%%%%%%%%%%%%%%%%
%%%%%%%%%%%%%%%%%%%%%%%%%%%%%%%%%%%%%%%%%%%

%%%%%%%%%%%%%%%%%%%%%%%%%%%%%%%%%%%%%%%%%%%%%%%%%%%%%%%%%%%%%%%%%%%%%%%%%%%%%%%%%%%%%%%%%%%%%%%%%%%%%%%%%%%%%%%%%%%%%%%%%%%%%%%%%%%%%%%%%%%%%%%%%%%%%%%%%%%%%%%%%%%%%%%%%%%%

\subsection{Trivializations and splittings}

\begin{proposition}\label{constants-matrixalgebra} For a $\delta-F-$central simple algebra $A$  with $n:=\sqrt{\mathrm{dim}_F(A)},$ the following statements are equivalent.
	\begin{enumerate}[(i)] \item $A$ is a split $\delta-F-$algebra. That is, $(A, \delta_A)\cong (\mathrm{M}_n(F), \delta^c)$ as $\delta-F-$algebras.
		\item $A^{\delta}\cong \mathrm{M}_n(F^{\delta})$ as $F^{\delta}-$algebras.
		\item $A^{\delta}$ has a subalgebra which is isomorphic to $\mathrm{M}_n(F^{\delta})$ as $F^{\delta}-$algebras.
		\end{enumerate}
	\end{proposition}

\begin{proof} $(i)\implies (ii).$ The ring of constants of $\mathrm{M}_n(F)$ under the derivation $\delta^c$ is the $F^{\delta}-$ algebra $\mathrm{M}_n(F^{\delta})$ and therefore, we must have $A^{\delta}\cong \mathrm{M}_n(F^{\delta})$ as $F^{\delta}-$algebras. 
	\item $(ii)\implies (iii)$ is obvious. 
	\item $(iii)\implies (i).$  Assume that $A^{\delta}$ has a $F^{\delta}-$subalgebra  $B$ isomorphic to $\mathrm{M}_n(F^{\delta}).$ Then $A^{\delta},$ as a $F^{\delta}-$vector space is of dimension $\geq n^2.$ Since the multiplication map $\mu_A: A^{\delta} \otimes_{F^\delta} F\to A$ is injective and $\mathrm{dim}_F(A )=n^2,$ we have $\mu_A$ is an isomorphism and that $A^{\delta}=B\cong \mathrm{M}_n(F^{\delta})$ as $F^{\delta}-$algebras. Thus, $$(\mathrm{M}_n(F), \delta^c)\cong \left(\mathrm{M}_n(F^{\delta})\otimes_{F^{\delta}} F, \delta^c\otimes \delta_F\right)\cong \left(A^{\delta} \otimes_{F^{\delta}} F, \theta\otimes  \delta_F\right)\cong (A, \delta_A),$$ where $\theta$ is the zero derivation and the isomorphisms are  $\delta-F-$algebra isomorphisms. 	\end{proof}
%%%%%%%%%%%%%%%%%%%%%%%%%%%%%%%%%%%%%%%%%%%%%%%%%%%%

The next three corollaries contain the proofs of Theorem \ref{Mainresult}(\ref{split-trivial-equivalence}) and (\ref{existence-minimal}).

\begin{corollary}\label{splitimpliestrivial}
	Let $A$ be a $\delta-F-$central simple algebra with $n:=\sqrt{\mathrm{dim}_F(A)}.$ If $A$ is a split $\delta-F-$algebra then $A$ is a trivial $\delta-F-$module.
	\end{corollary}

\begin{proof}
From Proposition \ref{constants-matrixalgebra},	 we know that $A^{\delta}\cong \mathrm{M}_n(F^{\delta})$ and therefore $\mathrm{dim}_{F^{\delta}}(A^{\delta}) =n^2.$ Thus $A$ is a trivial $\delta-F-$module.  
	\end{proof}
%%%%%%%%%%%%%%%%%%%%%%%%%%%%%%%%%%%%%%%%%%%%%%%%%%%%

\begin{corollary}	\label{trivial-split} Let $A$ be a $\delta-F-$central simple algebra. If  $A$ is a trivial $\delta-F-$ module  then $A^{\delta}$ is a central simple $F^{\delta}-$algebra and furthermore, if $F^{\delta}$ is an algebraically closed field then $A$ is a split $\delta-F-$central simple algebra.
\end{corollary}

\begin{proof}
Let $A$ be a trivial $\delta-F-$module. Then   $\mu_A: A^{\delta}\otimes_{F^{\delta}}F\to A$ is a $\delta-F-$algebra isomorphism. If $I$ is a nonzero ideal of $A^{\delta}$ then $I\otimes_{F^{\delta}}F$ is a nonzero ideal of $A^{\delta}\otimes_{F^{\delta}}F.$  Now since $A$ is simple, $\mu_A(I\otimes_{F^{\delta}}F)=A$ and therefore, $\mathrm{dim}_{F^{\delta}}(A^{\delta})=\mathrm{dim}_F(A)=\mathrm{dim}_F(I\otimes_{F^{\delta}}F) =\mathrm{dim}_{F^{\delta}}(I).$ In particular, $I=A^{\delta}$ and thus $A^{\delta}$ is a simple algebra. 

Let the center of $A^{\delta}$ be $Z.$ Then $Z$ is a $F^{\delta}-$vector space and $\mu_A(Z\otimes_{F^{\delta}}F)$ must be contained in the  center of $A,$ which is $F.$ Therefore, $\mathrm{dim}_F(\mu_A(Z\otimes_{F^{\delta}}F))=1$ and we have $Z=F^{\delta}.$	Thus, $A^{\delta}$ must be a $F^{\delta}-$central simple algebra. Since $F^{\delta}$ is assumed to be an algebraically closed field, we have $A^{\delta}\cong \mathrm{M}_n(F^{\delta})$ as $F^{\delta}-$algebras, where $n=\sqrt{\mathrm{dim}_F(A)}$. Now from Proposition \ref{constants-matrixalgebra}, we obtain that $A$ is a split $\delta-F-$algebra.\end{proof}

\begin{corollary}\label{trivial-split-PV} Let $A$ be a $\delta-F-$central simple algebra  and  $F^{\delta}$ be an algebraically closed field. Let $L$ be a $\delta-$field extension of $F$ with $F^{\delta}=L^\delta.$   Then, the $\delta-F-$ algebra  $A$ is split by the $\delta-$field $L$  if and only if $L$ contains a Picard-Vessiot extension $K$ for the $\delta-F-$module $A;$ in particular, $K$ also splits the $\delta-F-$central simple algebra $A.$ 
\end{corollary}

\begin{proof} Suppose that $L$ splits the $\delta-F-$algebra $A.$ Then the $\delta-L-$module $(A\otimes L, \delta_A\otimes\delta_L)$ is trivial and we have $\mathrm{dim}_{F^\delta}((A\otimes L)^{\delta_A \otimes \delta_L})=\mathrm{dim}_L(A\otimes L).$ Let $e_1,\dots, e_{n^2}$ be an $F-$basis of $A$ and $x_1,\dots, x_{n^2}$ be an $F^{\delta}-$basis of $(A\otimes L)^{\delta_A\otimes \delta_L}.$ Write \begin{equation}\label{expressionofconstants}x_j=\sum^{n^2}_{i=1}e_i\otimes\lambda_{ij}; \qquad \lambda_{ij}\in L\end{equation} 
	and consider the $\delta-$field $K$ generated over $F$ by $\{\lambda_{ij}\in L\ | 1\leq i, j\leq n^2 \}.$  Then since $K$ is a $\delta-$subfield of $L,$ $K^\delta=L^\delta.$ From Equation (\ref{expressionofconstants}), each $x_i$ belongs to $A\otimes K$ and therefore  $(A\otimes K, \delta_A\otimes \delta_K)$ is  a trivial $\delta-K-$module. Thus $K$ satisfies the equivalent definition of a Picard-Vessiot extension for the $\delta-F-$module $A$ given in  Section \ref{tannakianformalism}. Rest of the proof follows from Corollaries \ref{splitimpliestrivial} and \ref{trivial-split}.
\end{proof}

\begin{example} \label{pvminimality}
	Consider the rational function field $F:=\mathbb C(x)$ over the field of complex numbers with the derivation $\delta_F:=d/dx$ and the $\delta-F-$module $M:=M_2(F),$ where   $\delta_M:=  \delta^c+\mathrm{inn}_P$ for $$P:=\begin{pmatrix}\frac{1}{4x}&0\\ 0& -\frac{1}{4x}\end{pmatrix}.$$ 
	
	The $\delta-$field $E=F(\sqrt[4]{x})$ with $\delta_E=d/dx$ has a fundamental matrix \begin{equation}\label{fundmat-example}Z:=\begin{pmatrix}\sqrt[4]{x}&0\\0&\frac{1}{\sqrt[4]{x}}\end{pmatrix}\in \mathrm{GL}_n(E)\end{equation} 
	for the differential equation $\delta^c(Y)=PY.$ Since $\delta^c(Z)=PZ,$  the inner automorphism $X\mapsto ZXZ^{-1}$ of $M_2(E)$ becomes a  $\delta-E-$algebra automorphism between $(M_2(E), \delta^c)$ and $(M_2(E), \delta^c+\mathrm{inn}_P).$
	
	Now we shall show that the Picard-Vessiot extension for the $\delta-F-$module $M$ embeds in $E.$ Consider the $F-$basis of elementary matrices $\mathtt{E}_{ij},$ having $1$ in the $ij-$th position and $0$ elsewhere. Then $(\delta^c_E+\mathrm{inn}_P)(\mathtt{E}_{ij})=$ $\mathrm{inn}_P(\mathtt{E}_{ij})$ and the associated matrix differential equation of $M$ with respect to the (ordered) basis $\mathtt{E}_{11}, \mathtt{E}_{12}, \mathtt{E}_{21}, \mathtt{E}_{22}$ is given by the equation $$\delta^c(Y)=\begin{pmatrix}0& 0& 0& 0\\ 0& \frac{1}{2x}& 0 &0\\ 0&0&-\frac{1}{2x}&0\\ 0&0&0&0\end{pmatrix}Y.$$
	
	Clearly, the matrix	$$\begin{pmatrix}1& 0& 0& 0\\ 0& \sqrt{x}& 0 &0\\ 0&0&\frac{1}{\sqrt{x}}&0\\ 0&0&0&1\end{pmatrix}\in\mathrm{GL}_n(E),$$ whose entries generate the $\delta-$field $K:=F(\sqrt{x}),$  is a fundamental matrix of the above differential equation.  Thus $K$ is a Picard-Vessiot extension for the $\delta-F-$module $M,$ $F\subset K\subset E,$ and $$(M_2(K),\delta^c+\mathrm{inn}_P )\cong (M_2(K), \delta^c)$$ as $\delta-K-$algebras.
\end{example}

\begin{remark} We shall now build on Example \ref{pvminimality} to point out an inaccuracy (previously noted by Yidi Wang) in \cite[Proposition  2]{JM08}, which states that for any trace zero matrices $P, Q\in \mathrm{M}_n(K),$ the $\delta-K-$algebras $(\mathrm{M}_n(K), \delta^c_K+\mathrm{inn}_P)$ and $(\mathrm{M}_n(K), \delta^c_K+\mathrm{inn}_Q)$ are isomorphic if and only if $P$ is gauge equivalent to $Q$ over $K.$ From Example \ref{pvminimality},  we know that the $\delta-K-$algebras $(M_2(K), \delta^c_K+\mathrm{inn}_P)$ and $(M_2(K), \delta^c_K)$ are isomorphic. We shall now show that  $P$ is not gauge equivalent to the zero matrix. Assume on the contrary that there is a matrix $Z_1\in \mathrm{GL}_n(K)$ such that  $Z^{-1}_1\delta^c(Z_1)-Z^{-1}_1PZ_1=0.$ Then $\delta^c(Z_1)=PZ_1.$ Since $\delta^c(Z)=PZ,$  we have  $Z=Z_1C$ for some $C\in \mathrm{GL}_n(F^{\delta}).$ This would imply $\sqrt[4]{x}\in K=F(\sqrt{x}),$ a contradiction. 
	The error in the proof of \cite[Proposition  2]{JM08} is due to the assumption that one can always multiply the matrix $H\in\mathrm{GL}_n(K)$ (as defined in their proof) with a nonzero scalar matrix and obtain a matrix in $\mathrm{SL}_n(K).$ This assumption is incorrect as it holds only when $K$ contains an $n-$th root of $\mathrm{det}(H).$   
\end{remark}

\subsection{The $\delta-F-$Brauer Monoid.} We shall recall the definition of a $\delta-F-$Brauer monoid from \cite{Mag23}. Let $F$ be a $\delta-$field.  On the collection of all $\delta-F-$central simple algebras, consider the equivalence relation $$(A, \delta_A)\equiv (B, \delta_B) \quad \text{if} \ (\mathrm{M}_n(A), \delta^c)\cong (M_m(B), \delta^c) \ \text{for some}\ m,n\in \mathbb{N}\ \text{as }\ \delta-F-\text{algebras.} $$

The equivalence class of $\delta-F-$central simple algebras forms a monoid, called the $\delta-F$ \emph{Brauer monoid},  with respect to the tensor product $\otimes.$  The next Proposition is a special case of \cite[Theorem 1]{Mag23}, which we prove using the results  obtained in this section.
\begin{proposition} \label{torsionelements}
	Every torsion element of the $\delta-F-$Brauer monoid is represented by a $\delta-F-$algebra of the form $(B\otimes_{F^{\delta}} F, \theta\otimes \delta_F),$ where $B$ is central simple $F^{\delta}-$algebra with the zero derivation $\theta$. In particular,  if $F^{\delta}$ is an algebraically closed field then the $\delta-F-$Brauer monoid of $F$ is torsion free.
\end{proposition}

\begin{proof}
	Let $(A, \delta_A)$ represent a torsion element in the $\delta-F-$Brauer monoid with $n:=\sqrt{\mathrm{dim}_F(A)}.$ Then there is a positive integer $m$ such that $(A^{\otimes m}, \delta^{\otimes m}_A)$ is a split $\delta-F-$algebra. Now from Corollary \ref{splitimpliestrivial}, $(A^{\otimes m}, \delta^{\otimes m}_A)$ is a trivial $\delta-F-$module. Note that the natural map $$(A,\delta_A)\hookrightarrow A\otimes 1\otimes\cdots \otimes 1\subset (A^{\otimes m}, \delta^{\otimes m}_A)$$ is a $\delta-F-$algebra monomorphism. Now since   $(A^{\otimes m}, \delta^{\otimes m}_A)$ is a trivial $\delta-F-$module, $(A, \delta_A)$ must be a trivial $\delta-F-$module as well. This implies $\mu_A:A^{\delta}\otimes_{F^{\delta}} F\to A$ is an isomorphism of $\delta-F-$algebras. The last assertion of the proposition follows from Corollary \ref{trivial-split}.	\end{proof}
	
	%Thus, if $(A, \delta_A)$ represents torsion element of the $\delta-F-$Brauer monoid then $A^{\delta}\otimes_{F^{\delta}} F$ is isomorphic to $(A, \delta_A)$ as $\delta-F-$algebras under the map $\mu_A.$   If $F^{\delta}$ is algebraically closed then $A^{\delta}\cong \mathrm{M}_n(F^{\delta})$ as $F^\delta-$algebras and we obtain that $(\mathrm{M}_n(F), \delta^c)$ is isomorphic to $(A, \delta_A)$ as $\delta-F-$algebras.

In \cite[Theorem 5.1]{KS-2003}, it is shown that torsion elements in the $\delta-F-$Brauer monoid of the rational function field $F(x)$ with the derivation $x\frac{d}{dx}$ is split by a finite extension.  In the following theorem, we generalize  this result for an arbitrary 
$\delta-$field $F$ of characteristic zero.
\begin{theorem} Let $F$ be a $\delta-$field of characteristic zero.
 If a $\delta-F-$central simple algebra $A$ represents a torsion element in the $\delta-F-$Brauer monoid then there is a finite algebraic extension $K$ of $F$ such that  the $\delta-$field $K$ splits the $\delta-F-$algebra $A.$
\end{theorem}

\begin{proof} 
	By Proposition \ref{torsionelements}, we know that $A$ is a trivial $\delta-F-$module. Let $\overline{F}$ be an algebraic closure of $F.$ It is well-known that the derivation on $F$ extends uniquely to $\overline{F}$ and that 
	$\overline{F}^\delta=\overline{F^{\delta}}.$ Thus,  $\overline{F}^\delta$ is an algebraically closed field. Let $F_1$ be the compositum of $F$ and $\overline{F}^\delta.$ Since $A$ is trivial, the $\delta-F_1-$algebra  $(A\otimes F_1, \delta_A\otimes\delta_{F_1})$ is also trivial and therefore by Corollary  \ref{trivial-split}, it is a split $\delta-F_1-$algebra. Let $\phi:(A\otimes F_1,\delta_A\otimes \delta_{F_1})\rightarrow (\mathrm{M}_n(F_1), \delta^c)$ be a $\delta-F_1-$isomorphism and $\{v_1,v_2,\dots,v_{n^2}\}$ be an  $F-$basis of $A.$ Consider the field $K$ generated by $F$ and the entries of the matrices $\phi(v_i\otimes 1).$ Since $K\subseteq F_1,$ $K$ is a finite algebraic extension of $F$  and it is also a $\delta-$subfield of $F_1.$ As $K$ contains $\phi(v_i\otimes 1)$ for each $i=1,\dots, n^2,$  $K$ splits the $\delta-F-$algebra $A.$ 
\end{proof}

%%%%%%%%%%%%%%%%%%%%%%%%%%%%%%%%%%%%%%%%%%%%%%%%%%%%%%%%%%%%%%%%%%%%%%%%%%%%%%%%%%%%%%%%%%%%%%%%%%%%%%%%%%
\section{Projective Representations of Differential Galois Groups} \label{projrep}

  Let $A$ be a $\delta-F-$central simple algebra, $F^{\delta}$ be an algebraically closed field, $K$ be a Picard-Vessiot extension for some $\delta-F-$module $M$ such that $K$ splits the $\delta-F-$algebra $A$ and
$\mathscr{G}:=\mathscr G(K|F)$ be the $\delta-$Galois group of the $\delta-F-$module $M.$  Let $\psi_A: (\mathrm{M}_n(K), \delta^c)\to (A\otimes K, \delta_A\otimes \delta_K)$ be a $\delta-K-$algebra isomorphism. Then,  $\mathscr G-$action on $A\otimes K$  can be transported, through the isomorphism $\psi_A,$ to a $\mathscr{G}-$action on $\mathrm{M}_n(K):$    \begin{equation}\label{actionofGonM_n(K)}  \sigma \cdot (a_{ij})=(\psi^{-1}_A\circ(1\otimes \sigma)\circ \psi_A) (a_{ij}), \quad (a_{ij})\in \mathrm{M}_n(K).\end{equation} 
	We shall denote this $\mathscr G-$algebra $\mathrm{M}_n(K)$ by $(\mathrm{M}_n(K), \mathscr G_{\psi_A}).$ In the next few paragraphs, we shall describe this action in detail and in the process, we shall also obtain a projective representation of the $\delta-$Galois group.
	
	Let $\mathscr S(A):=(A\otimes K)^{\delta_A\otimes \delta_K}.$  We recall that the $\mathscr G-$action on $A\otimes K$ stabilizes $\mathscr S(A)$ and that $\psi_A$ also restricts to an isomorphism between the $\mathscr G-$algebras $\left(\mathrm{M}_n(F^{\delta}), \mathscr G_{\psi_A}\right)$ and $\left(\mathscr S(A), 1\otimes \mathscr{G}\right)$ and that  the $\mathscr G-$action on $\mathscr S(A)$ gives a representation of the linear algebraic group $\mathscr G:$ $$\rho_A: \mathscr{G}\rightarrow\mathrm{GL}(\mathscr S(A)), \quad \rho_A(\sigma)(x)=(1\otimes \sigma)(x).$$  Note that $\rho_A$ need not be faithful as $K$ is not assumed to be the Picard-Vessiot extension for the $\delta-F-$module $A.$ Since for each $\sigma\in \mathscr{G},$ $\rho_A(\sigma)$ is an $F^{\delta}-$algebra automorphism of $\mathscr S(A)$ and  since the subgroup $\mathrm{Aut(\mathscr S(A))}$ of all  $F^{\delta}-$algebra automorphisms is Zariski closed in $\mathrm{GL}(\mathscr S(A)),$ we have  the following (commutative diagram of) morphisms of algebraic groups:
	\begin{equation*}\label{algebraicgrouprep}
		\begin{tikzcd}
			\mathscr{G}\arrow[dr, rightarrow, "\rho_A"] \arrow[d, rightarrow]& \\
			\mathrm{Aut}(\mathscr S(A))\arrow[r, hookrightarrow,"inc"]&\mathrm{GL}(\mathscr S(A))
		\end{tikzcd}
	\end{equation*}
	
Since $\left(\mathrm{M}_n(F^{\delta}), \mathscr G_{\psi_A}\right)\overset{\psi_A}{\cong} \left(\mathscr S(A), 1\otimes \mathscr G\right),$  we therefore have  \begin{equation}\label{algebraicgrouprepmnc}
		\begin{tikzcd}
			\mathscr{G}\arrow[dr, rightarrow, "\tilde{\rho}_A"] \arrow[d, rightarrow]& \\
			\mathrm{Aut}(\mathrm{M}_n(F^{\delta}))\arrow[r, hookrightarrow,"inc"]&\mathrm{GL}(\mathrm{M}_n(F^{\delta}))
		\end{tikzcd}
	\end{equation}
	where $\tilde{\rho}_A(\sigma)(a_{ij})=\psi^{-1}_A(1\otimes \sigma) \psi_A(a_{ij}),$ for $(a_{ij})\in \mathrm{M}_n(F^\delta).$
  Since the automorphisms of $\mathrm{M}_n(F^{\delta})$ are the inner automorphisms, $\mathrm{Aut}(\mathrm{M}_n(F^{\delta}))\cong \mathrm{PGL}_n(F^{\delta})$ as algebraic groups. Thus we obtain a morphism of algebraic groups (projective representation) \begin{equation}\label{projectiverepresentation}\varphi_A: \mathscr G\rightarrow \mathrm{PGL}_n(F^{\delta})\end{equation} by composing $\tilde{\rho}_A$ with the isomorphism  $\mathrm{Aut}(\mathrm{M}_n(F^{\delta}))\cong \mathrm{PGL}_n(F^{\delta}).$ If $\pi: \mathrm{GL}_n(F^{\delta})\to \mathrm{PGL}_n(F^{\delta})$ is the canonical projection then  for every $\sigma\in \mathscr G$ and $(a_{ij})\in \mathrm{M}_n(F^{\delta}),$ 
	\begin{equation}\label{rhoasconjugation}\varphi_A(\sigma)(a_{ij}):=\tilde{\rho}_A(\sigma)(a_{ij})=(c_{ij\sigma})(a_{ij}) (c_{ij\sigma})^{-1},\end{equation} where  $(c_{ij\sigma})\in \mathrm{GL}_n(F^{\delta})$ is such that $\pi(c_{ij\sigma})=\varphi_A(\sigma).$   We note here that the projective representation $\varphi_A$ that we have obtained is the same as the one constructed in \cite[Section 3]{JM08} for matrix $\delta-F-$algebras.

\begin{proof}[Proof of Theorem \ref{proj-equivalence}]
	Let $\delta-\mathrm{CSA}_{K|F,\ n}$ denote the collection of all  $\delta-F-$algebra isomorphism classes of central simple algebras of dimension $n^2$ that are split by the Picard-Vessiot extension $K$ of $F.$ Let $\mathrm{PRep}_{\mathscr G,\ n}$ be the collection of all inequivalent projective representations of $\mathscr G:=\mathscr G(K|F)$ to $\mathrm{PGL}_n(F^\delta).$ Let $$\Gamma: \delta-\mathrm{CSA}_{K|F,\ n}\to \mathrm{PRep}_{\mathscr G,\ n}; \quad \Gamma([A])= [\varphi_A],$$ where $A$ is a representative of  an arbitrary equivalence class in $\delta-\mathrm{CSA}_{K|F,\ n}$ and $\varphi_A$ is the representation defined in  Equation (\ref{projectiverepresentation}). 
	
	We shall first show that $\Gamma$ is well-defined.  Let $\alpha: A\to B$ be an isomorphism  of $\delta-F-$central simple algebras. Then $\alpha\otimes 1: (A\otimes K, 1\otimes \mathscr G)\to (B\otimes K, 1\otimes \mathscr G)$ is a $\mathscr G-$equivariant isomorphism and also a $\delta-K-$algebra isomorphism between $(A\otimes K, \delta_A\otimes \delta_K)$ and $(B\otimes K, \delta_B\otimes \delta_K).$  Let $\beta: M_n(K)\to M_n(K)$ be the isomorphism obtained by composing the following isomorphisms: $$M_n(K)\overset{\psi_A}{\to} A\otimes K\overset{\alpha\otimes 1}{\to} B\otimes K\overset{\psi^{-1}_B}{\to}M_n(K).$$ Observe that $\beta$ is a $\mathscr G-$ equivariant isomorphism between the $\mathscr G-$algebras $(M_n(K),\mathscr{G}_{\psi_A})$ and $(M_n(K),\mathscr{G}_{\psi_B})$ and that $\beta$ is also a $\delta-K-$algebra automorphism of $(M_n(K), \delta^c).$ This implies $\beta$ restricts to a $\mathscr{G}-$algebra isomorphism between $(M_n(F^{\delta}),\mathscr{G}_{\psi_A})$ and $(M_n(F^{\delta}),\mathscr{G}_{\psi_B}).$ Thus $$\beta \tilde{\rho}_A(\sigma) \beta^{-1}=\tilde{\rho}_B(\sigma),\quad \text{for all} \ \sigma\in \mathscr G.$$ Now from Equation (\ref{rhoasconjugation}), we obtain that $\varphi_A$ and $\varphi_B$ are equivalent projective representations.

	 $\Gamma$ is injective: Assume on the contrary that there are $\delta-F-$algebras $A$ and $B$ in $\delta-\mathrm{CSA}_{K|F,\ n}$ representing two distinct classes in $\delta-\mathrm{CSA}_{K|F,\ n}$ and that their corresponding projective representations  $\varphi_A$ and $\varphi_B$ are equivalent. Then,  there is a matrix $H\in \mathrm{GL}_n(F^{\delta})$ such that $\varphi_A(\sigma)=\pi(H)\varphi_B(\sigma)\pi(H)^{-1}$ for all $\sigma \in G.$ From the natural conjugation action of $\mathrm{PGL}_n(F^{\delta})$ on $\mathrm{M}_n(F^{\delta}),$ the  group  $\mathscr G$ acts on $\mathrm{M}_n(F^{\delta})$ as automorphisms of $F^{\delta}-$algebras through the representations $\varphi_A$ and $\varphi_B$ as follows: For $X\in \mathrm{M}_n(F^{\delta}),$ \begin{align*}\varphi_A(\sigma) (X)&=Z_{\varphi_A(\sigma)} X Z^{-1}_{\varphi_A(\sigma)}, \quad Z_{\varphi_A(\sigma)}\in  \pi^{-1}(\varphi_A(\sigma))\\  \varphi_B(\sigma) (X)&=Z_{\varphi_B(\sigma)} X Z^{-1}_{\varphi_B(\sigma)},\quad Z_{\varphi_B(\sigma)} \in \pi^{-1}(\varphi_B(\sigma)).\end{align*}
	
	Now, from Equation (\ref{rhoasconjugation}), it is easily seen that  $$\left(\mathrm{M}_n(F^{\delta}), \mathscr G_{\psi_A}\right)\to \left(\mathrm{M}_n(F^{\delta}), \mathscr G_{\psi_B}\right);\quad X\to H^{-1}XH$$ is an isomorphism of $\mathscr G-$algebras.  Since $$\mathscr S(A)\overset{\psi^{-1}_A}{\cong} \mathrm{M}_n(F^{\delta})\overset{\psi_B}{\cong} \mathscr S(B),$$ we have a $\mathscr G-$algebra isomorphism $\phi:(\mathscr S(A), 1\otimes \mathscr G)\to(\mathscr S(B), 1\otimes \mathscr G).$  Then $$\phi\otimes 1: \left(\mathscr S(A)\otimes_{F^{\delta}} K, 1\otimes \mathscr G\otimes_{F^{\delta}} \mathscr G\right)\to \left(\mathscr S(B)\otimes_{F^{\delta}} K,1\otimes \mathscr G\otimes_{F^{\delta}} \mathscr G\right)$$ is an isomorphism of $\mathscr G-$algebras as well as of $\delta-K-$algebras with derivation $\theta\otimes \delta_K.$  Using the $\mathscr G-$equivariant $\delta-K-$isomorphisms $\mu_A$ and $\mu_B,$ we  see that $A\otimes K$ is isomorphic to $B\otimes K$ as both $K-$algebras as well as  $\mathscr G-$algebras.  Finally, taking $\mathscr G-$invariance and using the fact that $K^{\mathscr G}=F$, we obtain that the $\delta-F-$algebras $A$ and $B$ are isomorphic, which contradicts our assumption. Thus $\Gamma$ must be injective.

	 $\Gamma$ is surjective:  Let $\varphi: \mathscr G\rightarrow \mathrm{Aut}(\mathrm{M}_n(F^{\delta}))\cong \mathrm{PGL}_n(F^{\delta})$ be a projective representation. Since $\mathrm{Aut}(\mathrm{M}_n(F^{\delta}))$ is a closed algebraic subgroup of $\mathrm{GL}(\mathrm{M}_n(F^{\delta})),$ we have $$\mathscr G\overset{\varphi}{\rightarrow} \mathrm{Aut}(\mathrm{M}_n(F^{\delta}))\overset{\mathrm{inc}}{\hookrightarrow} \mathrm{GL}(\mathrm{M}_n(F^{\delta})).$$
	 Let $\{\{M\}\}$ be the Tannakian category of a $\delta-F-$module $M$ whose Picard-Vessiot extension is $K.$ Since $\{\{M\}\}$ is equivalent to $\mathrm{Rep}_{\mathscr G},$   there is a  $\delta-F-$module $N\in \{\{M\}\}$ of dimension $n^2$ over $F$  and  an $F^{\delta}-$vector space isomorphism  $\psi: \mathrm{M}_n(F^{\delta})\to \mathscr S(N)$ such that the representation $\mathrm{inc}\circ\varphi$ is equivalent to  $\rho_{N}: \mathscr G\to \mathrm{GL}(\mathscr S(N)).$  Note that the $\mathscr{G}$ action on $\mathscr S(N)$ given by $\rho_{N}$ is the restriction of the $\mathscr G$ action on $N\otimes K;$   $x\otimes k\mapsto x\otimes \sigma(k)$ to $\mathscr S(N).$
	 
	 Using the isomorphism $\psi,$ we consider  $\mathscr S(N)$ as a central simple (matrix) $F^{\delta}-$algebra. Then $(\mathscr S(N)\otimes_{F^{\delta}} K, \theta\otimes_{F^{\delta}} \delta_K)$ is a $\delta-K-$central simple algebra. Moreover, now,  $\mathscr G-$acts on $\mathscr S(N)$  by algebra automorphisms. Then, the  $\mathscr G-$equivariant $\delta-K-$module isomorphism $$\mu_N: \mathscr S(N)\otimes_{F^{\delta}} K\to N\otimes K$$ can be made into an isomorphism of $\delta-K-$algebras by transporting the algebra structure of $\mathscr S(N)\otimes_{F^{\delta}} K$ on to $N\otimes K$, where the $\mathscr{G}$ actions are as described in Equations \ref{actiononsinglecoordinate} and \ref{actiononbothcoordinates}.  By taking $\mathscr G-$invariance, we obtain  $$N\cong (N\otimes K)^{\mathscr G}\cong \left(\mathscr S(N)\otimes_{F^{\delta}} K\right)^{\mathscr G}.$$ 
	 
	 Since $N\otimes K$ is a central simple $K-$algebra and   $$N\otimes K\cong N\otimes F\otimes K\cong (N\otimes K)^\mathscr G\otimes K,$$ we obtain that $(N\otimes K)^\mathscr G$  has a structure of a central simple $F-$algebra (see the Proof of Corollary \ref{trivial-split}).  Thus,  $N$ can be viewed as a $\delta-F-$central simple algebra. Since $(N, \delta_N)$ is also split by the $\delta-$field $K,$ it represents a class in $\delta-\mathrm{CSA}_{K|F,\ n}.$ By definition, $\Gamma ([N])=[\rho_N].$  Since $\rho_N$ is equivalent to $\varphi,$ we obtain that $\Gamma$ is surjective.\end{proof}

	\begin{remarks*}
		
	\item \begin{enumerate}[(i)]
		\item Let $K$ and $\mathscr G$ be as in Theorem \ref{proj-equivalence},  $\mathscr G$ act on $\mathrm{GL}_n(K)$ coordinate-wise and consider the pointed set $\mathrm H^1(\mathscr G, \mathrm{GL}_n(K))$ of cohomologous $1-$cocycles. The inclusion map $i: \mathrm{GL}_n(F^\delta)\to \mathrm{GL}_n(K)$ and the projection map $\pi: \mathrm{GL}_n(K)\to \mathrm{PGL}_n(K)$ induces natural maps of pointed sets $i^*: \mathrm{H}^1(\mathscr G, \mathrm{GL}_n(F^\delta))\to \mathrm{H}^1(\mathscr G, \mathrm{GL}_n(K))$ and $\pi^*: \mathrm{H}^1(\mathscr G, \mathrm{GL}_n(K))\to \mathrm{H}^1(\mathscr G, \mathrm{PGL}_n(K)).$  Since $\mathscr G$ fixes elements (of $F$ and therefore) of $F^\delta,$ we observe that $\mathrm{H}^1(\mathscr G, \mathrm{GL}_n(F^\delta))$ consists of  equivalence classes of group homomorphisms from $\mathscr G$ to $\mathrm{GL}_n(F^\delta)$ and $\mathrm{PRep}_{\mathscr G, \ n}$  can be identified with the subclasses of algebraic group homomorphisms from $\mathscr G$ to $\mathrm{PGL}_n(F^\delta).$ We also have the following commutative diagram \[\begin{tikzcd}\mathrm{H}^1(\mathscr G,\mathrm{GL}_n(F^\delta)) \arrow{r}{\pi^*} \arrow[swap]{d}{i^*} & \mathrm{H}^1(\mathscr G,\mathrm{PGL}_n(F^\delta)) \arrow{d}{i^*} \\	\mathrm{H}^1(\mathscr G,\mathrm{GL}_n(K)) \arrow{r}{\pi^*} & \mathrm{H}^1(\mathscr G,\mathrm{PGL}_n(K))	\end{tikzcd}	\]
		In the context of matrix $\delta-$algebras, the projective representation of $\mathscr G$ considered in \cite{JM08} and the one we are considering here is one and the same,  \cite[Theorem 2]{JM08} becomes applicable. Thus we obtain that  for $\varphi\in \mathrm{PRep}_{\mathscr G, \ n}\subset \mathrm{H}^1(\mathscr G,\mathrm{PGL}_n(F^\delta)),$ $\Gamma^{-1}(\varphi)$ is a class represented by a matrix $\delta-F-$algebra if and only if $i^*(\varphi)$  is the trivial class in $\mathrm{H}^1(\mathscr G,\mathrm{PGL}_n(K));$ that is,    there is an element $H\in \mathrm{GL}_n(K)$ such that for all $\sigma\in\mathscr G,$ $\varphi(\sigma)=\pi(H^{-1}\sigma(H)),$ where $\pi:  \mathrm{GL}_n(K)\to \mathrm{PGL}_n(K)$ is the canonical projection. \\
		\item  Let $K$ and $\mathscr G$ be as in Theorem \ref{proj-equivalence}. If $\varphi\in \mathrm{PRep}_{\mathscr G, n}$ lifts to a linear representation of algebraic groups $\tilde{\varphi}: \mathscr G\to \mathrm{GL}_n(F^\delta),$ then since  there is an $N\in \{\{M\}\}$  such that $\tilde{\varphi}$ is equivalent to the representation $\rho_N: \mathscr G\to \mathrm{GL}(\mathscr S(N)),$ we obtain that $\varphi$ and $\pi\circ \rho_N$ are equivalent projective representations. Now, from Equation (\ref{repcocycle}), we have  $H\in \mathrm{GL}_n(K)$ such that $\rho_N(\sigma)=H^{-1}\sigma(H).$ This implies $(\pi\circ \rho_N)(\sigma)=\pi(H^{-1}\sigma(H)),$ for all $\sigma\in \mathscr G.$ Thus if a projective representation $\varphi$  lifts to a linear representation then any representative of $\Gamma^{-1}(\varphi)$ is a matrix $\delta-F-$algebra. \\
		
			%\item 

		\item Let $A$ be a $\delta-F-$central simple algebra.  Define $$\mathrm{deg}^{\delta}_{\mathrm{sp}}(A):= \mathrm{min}\{\mathrm{tr.deg}(L|F)\ | \ (L, \delta_L)\supseteq (F, \delta_F) \ \text{splits}\ (A, \delta_A)\ \text{and}\ L^{\delta}=F^{\delta}\},$$ where $\mathrm{tr.deg}(L|F)$ denote the field transcendence degree of $L$ over $F.$   It is shown in \cite[Corollary 5.2]{tsui-wang} that $\mathrm{deg}^{\delta}_{\mathrm{sp}}(A)\leq n^2-1,$ where $[A: F]=n^2.$ From Theorem \ref{Mainresult} (\ref{existence-minimal}), we obtain that $$\mathrm{deg}^{\delta}_{\mathrm{sp}}(A)=\mathrm{tr.deg}(K|F),$$ where $K$ is the Picard-Vessiot extension for the $\delta-F-$module $A.$  Furthermore, if we  consider the faithful representation described in the  diagram (\ref{algebraicgrouprepmnc}) and in the Equation (\ref{projectiverepresentation})
		 $$\mathscr{G}(K|F)\hookrightarrow \mathrm{PGL}_n(F^\delta)\hookrightarrow \mathrm{GL}_{n^2}(F^{\delta})$$ then $\mathrm{dim}(\mathscr G(K|F))\leq n^2-1$ and  by the fundamental theorem of differential Galois theory (\cite[Corollary 1.30]{MvdP03}),  $$\mathrm{deg}^{\delta}_{\mathrm{sp}}(A)=\mathrm{tr.deg}(K|F)=\mathrm{dim}(\mathscr{G}(K|F))\leq n^2-1.$$ 
		 \end{enumerate}
		\end{remarks*}

The next proposition describes a relationship  between $\mathscr{G}-$stable ideals of $(\mathscr S(A), 1\otimes \mathscr{G})$ and $\delta-$ideals of a $\delta-F-$algebra $A$ that will play an important role in the proofs of Theorem \ref{Mainresult} (\ref{reductivecase}) and (\ref{liouvilliancase}).

\begin{proposition} ($\mathrm{cf.}$ \cite[Corollary 2.35]{MvdP03}) \label{bijectivecorrespondence-ideals} Let $A$ be a $\delta-F-$central simple algebra, $F^{\delta}$ be an algebraically closed field, $K$ be a Picard-Vessiot extension for the $\delta-F-$module $A$ and $\mathscr G$ be its $\delta-$Galois group.  Then, the fibre functor $\mathscr S$ 
	$$\mathscr S(N)=(N\otimes K)^{\delta_N\otimes \delta_K}$$
	that establishes an equivalence between the categories $\{\{A\}\}$ and $\mathrm{Rep}_\mathscr G$ also establishes an inclusion preserving bijective correspondence between the $\delta-$ideals of $(A, \delta_A)$ and the $\mathscr G-$stable ideals of $(\mathscr S(A), 1\otimes\mathscr G).$   \end{proposition}
\begin{proof} 
	If $J$ is a $\delta-$ideal of $A$ then it is also a $\delta-F-$submodule of $A$ and therefore, $\mathscr S (J)\in \mathrm{Rep}(\mathscr G)$ is $\mathscr{G}-$stable. Also, $J\otimes K$ is a  $\delta-$ideal of $A\otimes K$ and thus $\mathscr S(J)$ is an ideal of $\mathscr S(A).$
	
	To prove the converse, we consider the $\mathscr G-$ equivariant $\delta-K-$algebra isomorphism $\mu_A:\mathscr S (A)\otimes_{F^{\delta}} K\to A\otimes K$ and  take $\mathscr{G}-$invariance to obtain  the following $\delta-F-$algebra isomorphisms \begin{equation}\label{inversefunctor}(\mathscr S (A)\otimes_{F^{\delta}} K)^\mathscr G\overset{\mu_A}{\cong}(A\otimes K)^\mathscr G \cong A.\end{equation}
	
	Now let $I$ be a $\mathscr G-$stable ideal of $\mathscr S(A)$ and $\tilde{\mu}_A:(\mathscr S (A)\otimes_{F^{\delta}} K)^\mathscr G\to A$ be the $\delta-F-$isomorphism obtained by taking the  composition of above isomorphisms. Then $I\in \mathrm{Rep}_\mathscr G$ and therefore, $I=\mathscr S(J)$ for some $\delta-F-$submodule $J$ of $(A, \delta_A).$ We only need to show that $J$ is an ideal of $A.$  Since $\mathscr S(J)$ is an ideal of $\mathscr S(A),$ we have $\mathscr S(J)\otimes_{F^{\delta}} K$ is a $\delta-$ideal of $\mathscr S (A)\otimes_{F^{\delta}} K.$ Then, one can simply replace $A$ by $J$ in Equation (\ref{inversefunctor}) and obtain $$\tilde{\mu}_A\left(\mathscr S(J)\otimes_{F^{\delta}} K\right)\cong J.$$ Thus $J$ is an ideal of $A.$ \end{proof}

\section{Complete reduciblity and solvability} \label{reducible-solvable}

In this section we shall prove Theorem \ref{Mainresult} (\ref{reductivecase}) and (\ref{liouvilliancase}). Let $(F^\delta)^n$ be the $n-$dimensional $F^\delta-$vector spaces of column vectors and  $C\in\mathrm{GL}_n(F^\delta).$ For $v\in (F^\delta)^n,$  $Cv$ denotes the usual matrix multiplication of $C$ and $v.$ For a subspace $W$ of $(F^\delta)^n,$ $CW$ denotes the image of $W$ under the matrix multiplication by $C.$ For any subset $I$ of $\mathrm{M}_n(F^\delta),$ $C\cdot I:= CIC^{-1}$ denotes the image of $I$ under the conjugation by $C.$

\begin{proposition}\label{Phiequivalence}
	There is a bijection $\Phi$ between the collection $\mathcal I$ of all right ideals of $\mathrm{M}_n(F^\delta)$ and the collection $\mathcal S$ of all $F^\delta-$subspaces  of $(F^\delta)^n$ having the following properties: For right ideals $I, J,$ 
	\begin{enumerate}[(i)]
		\item  if $I\subseteq J$ then  $\Phi(I)\subseteq \Phi(J).$
		\item $\Phi(I\cap J)=\Phi(I)\cap \Phi(J).$ 
		\item $\Phi(I+J)=\Phi(I)+\Phi(J)$ and if $I+J=I\oplus J$ then $\Phi(I\oplus J)=\Phi(I)\oplus \Phi(J).$ 
		\item $\Phi(C\cdot I)=C\Phi(I).$
	\end{enumerate}
\end{proposition}

\begin{proof}
	Let $\mathtt{e}_{1}, \dots, \mathtt{e}_n$ be the standard basis of $(F^\delta)^n.$ Then, the mapping $\Phi: \mathcal I\to \mathcal S$ that sends an ideal $I$ to the $F^\delta-$subspace of all columns of all matrices belonging to $I,$ that is,  $$\Phi(I):=\{(a_{ij})\mathtt{e}_{s}\ | \ \text{for all}\  1\leq s\leq n, (a_{ij})\in I\},$$  can be easily seen to be  bijective and the proofs of all four parts of the proposition follows immediately from the definition of our $\Phi.$ 
\end{proof}

We are now ready to prove Theorem \ref{Mainresult}(\ref{reductivecase}) and (\ref{liouvilliancase}). In the proofs,   $\mathscr H$ denote the group $\pi^{-1}(\varphi_A(\mathscr G)),$ where $\varphi_A$ and $\pi$ are as described in Section \ref{projrep}.

\begin{proof}[Proof of Theorem \ref{Mainresult}(\ref{reductivecase})] First we shall make the following general observation. Since $K$ is the Picard-Vessiot extension of $F$ for the $\delta-F-$module $A,$ we have that $\rho_A$ is faithful and consequently that $\varphi_A$ is a faithful map of algebraic groups.  Since $F^\delta$ is of characteristic zero, $\varphi_A: \mathscr G\to \mathscr \varphi_A(\mathscr G)$ is an isomophism of  algebraic groups. This observation in turn implies that the  exact sequence $$1\to \mathrm{G}_m\to \mathrm{GL}_n(F^{\delta})\to \mathrm{PGL}_n(F^{\delta})\to 1,$$  when restricted to   the closed subgroup $\varphi_A(\mathscr{G})$ of $\mathrm{PGL}_n(F^{\delta}),$ yields 
	  the exact sequence of algebraic groups $$1\to \mathrm{G}_m\to \mathscr H \to \mathscr G\to 1.$$ 
	  
	  Suppose that $\mathscr G$ is reductive.  Since $\mathrm{G}_m$ and $\mathscr{G}$ are (linearly) reductive,  $\mathscr H$ is (linearly) reductive (\cite[Proposition 3.4]{Benedict-09}). The natural action of  $\mathscr H$ on $(F^{\delta})^n$ given by \begin{equation}\label{actiononideals-subspaces} \mathtt{e}_i\mapsto (c_{ij})\mathtt e_i\end{equation} makes $(F^{\delta})^n$ a completely reducible $\mathscr H-$module. Thus, we have \begin{equation}\label{decompvectorspace}(F^{\delta})^n=\bigoplus^l_{i=1} V_i,\end{equation} where each $V_i$ is an irreducible $\mathscr H-$module. 
	
Let  $\mathcal I,$ $\mathcal S$ and $\phi$ be as in Proposition $\ref{Phiequivalence}.$   Recall that in the $\mathscr G-$algebra $(\mathrm{M}_n(F^\delta), \mathscr G_{\psi_A}),$ for each $\sigma\in \mathscr G,$ there is a matrix $(c_{ij\sigma})\in \mathscr H$ such that $\sigma \cdot (a_{ij})=(c_{ij\sigma})(a_{ij})(c_{ij\sigma})^{-1}$ for all $(a_{ij})\in \mathrm{M}_n(F^\delta).$ For $i=1,2,\dots, l,$ let $\tilde{I}_i\in \mathcal I$ be such that $\Phi(\tilde{I}_i)=V_i.$  Now from Proposition $\ref{Phiequivalence},$ we obtain $$\Phi\left(\sigma\cdot\tilde{I}_i\right)=\Phi\left((c_{ij\sigma})\tilde{I}_i(c_{ij\sigma})^{-1}\right)=(c_{ij\sigma}) \Phi(\tilde{I}_i)=(c_{ij\sigma})V_i=V_i=\Phi(\tilde{I}_i).$$
Since $\Phi$ is bijective, we have $\sigma\cdot \tilde{I}_i=\tilde{I}_i$ for each $i=1,2,\dots, l.$ That is, $\tilde{I}_i$ is  $\mathscr G-$stable right ideal of the $\mathscr G-$algebra $(\mathrm{M}_n(F^\delta), \mathscr G_{\psi_A})$. Furthermore,  since each $V_i$ is an  irreducible $\mathscr H-$module, each $\tilde{I}_i$ must be minimal among all $\mathscr G-$stable right ideals of $\mathrm{M}_n(F^\delta).$  Therefore,  applying $\Phi^{-1}$ and then $\psi_A$ to Equation (\ref{decompvectorspace}), we obtain  \begin{equation}\label{tensordescent}\mathscr S(A):=(A\otimes K)^{\delta_A\otimes \delta_K}=\bigoplus^l_{i=1}I_i,\end{equation}   
	where each $I_i$ is minimal among all $\mathscr G-$stable right ideals of $(A\otimes K)^{\delta_A\otimes \delta_K}.$   Now, from Proposition \ref{bijectivecorrespondence-ideals},  we obtain that $A$ is  an internal direct sum of minimal $\delta-$right ideals of $A.$
	
	Converse is proved, by simply retracing the above arguments, as follows.  Suppose that $$A= \oplus^l_{i=1}\mathfrak a_i,$$ as $\delta-F-$modules, where each $\mathfrak{a}_i$ is a minimal $\delta-$right ideal of $A.$ We then apply $\mathscr S$ followed by $\psi^{-1}_A$ and obtain $$\mathrm{M}_n(F^{\delta})=  \psi^{-1}_A(\mathscr S(A))= \oplus^l_{i=1}I_i,$$ where 
	$I_i=\psi^{-1}_A(\mathscr S(\mathfrak a_i)).$ Since $\mathfrak{a}_i$ is a minimal $\delta-$right ideal of $A,$ we have $I_i$ is  minimal among all $\mathscr{G}-$stable right ideals of  $\mathrm{M}_n(F^{\delta}).$ Now, 
	applying $\Phi,$ we obtain that $$(F^{\delta})^n=\oplus^l_{i=1}V_i,$$  where each $\Phi(I_i)=V_i$ is a faithful irreducible $\mathscr H-$module under the action defined in Equation (\ref{actiononideals-subspaces}). This makes $\mathscr H$ a linearly reductive linear algebraic group. Since   $\mathscr G$ is a quotient of $\mathscr H$ by the normal subgroup of scalar matrices of $\mathrm{GL}_n(F^{\delta})$ and the characteristic of $F^{\delta}$ is zero,  we obtain that $\mathscr G$ is reductive. \end{proof}
	
	\begin{proof}[Proof of Theorem \ref{Mainresult}(\ref{liouvilliancase})]   Let $\mathscr G^0$ be a solvable algebraic group and $F^0$ be the algebraic closure of $F$ in $K.$ Then, $F^0=E^{\mathscr G^0}$ (\cite[Proposition 1.34]{MvdP03}).   First assume that $\mathscr{G}^0=\mathscr{G},$ that is, $\mathscr G$ is connected so that $F^0=F.$   From the exact sequence of algebraic groups $$0\to \mathrm{G}_m\to \mathscr H\to \mathscr G\to 0,$$ we observe that $\mathscr H$ is a connected solvable group if and only if $\mathscr G$ is. Thus, we have a chain of $\mathscr H-$stable subspaces of $(F^{\delta})^ n$ $$V_1\subsetneq V_2\subsetneq\cdots\subsetneq V_n=(F^{\delta})^n$$  
such that $\mathrm{dim}_{F^\delta}(V_j)=j$ for $j=1,2,\dots, n.$  
	
	 We apply $\Phi^{-1}$ and then $\psi_A$  to obtain a chain of $\mathscr{G}-$stable right ideals $I_i$ of $(\mathscr S(A), 1\otimes \mathscr G)$ $$I_1\subsetneq I_2\subsetneq\cdots\subsetneq \mathscr S(A).$$
	Since $\mathrm{dim}_{F^\delta}(V_j)=j,$ we must have  $\mathrm{dim}_{F^\delta}(I_j)=jn.$ From Proposition \ref{bijectivecorrespondence-ideals}, we now obtain a chain of  $\delta-$ideals $J_j$ of $A$ $$J_1\subsetneq J_2\subsetneq\cdots\subsetneq A$$
	with  $\mathrm{dim}_F(J_j)=jn.$  In particular, $J_1$ is a right ideal of $A$ with $\mathrm{dim}_F(J_1)=n$  and therefore, $A$ is a split $F-$algebra.
	
	Now, suppose that $\mathscr G\neq \mathscr{G}^0$ and that $\mathscr G^0$ is solvable. Since $K$ is a Picard-Vessiot extension of $F^0$ for the  $\delta-{F^0}-$module $(A\otimes F^0)$ and the $\delta-$Galois group $\mathscr{G}(K|F^0)=\mathscr{G}^0$ is a connected solvable group, we obtain from the previous case that $A\otimes F^0$ is a split $F^0-$algebra.
	
	To prove the converse, we first note that it does no harm to assume that $A$ is a split $F-$algebra and that the $\delta-$Galois group $\mathscr G$ of the $\delta-F-$module $A$ is connected. Suppose that there is chain of $\delta-$ideals of $A$ $$J_1\subsetneq J_2\subsetneq\cdots\subsetneq A$$ with $\mathrm{dim}_F(J_j)=jn.$ We shall now show that $\mathscr G$ is solvable. From Proposition \ref{bijectivecorrespondence-ideals}, there is a chain of $\mathscr{G}-$stable ideals of $\mathscr S (A):$ $$I_1\subsetneq I_2\subsetneq\cdots\subsetneq \mathscr S(A)$$
	with  $\mathrm{dim}_{F^\delta}(I_j)=jn.$ Now applying $\psi^{-1}_A$ and then $\Phi,$ we obtain $\mathscr H-$stable subspaces $$V_1\subsetneq V_2\subsetneq\cdots\subsetneq V_n=(F^{\delta})^n$$ with $\mathrm{dim}_{F^\delta}(V_j)=j.$ Thus $\mathscr H=\pi^{-1}(\varphi_A(\mathscr{G}))$ is solvable  and therefore, $\mathscr G$ is solvable.
\end{proof}

\begin{remark} If $A$ is a $\delta-F-$central division algebra then since $A$ has no nontrivial right ideals and therefore by Theorem \ref{Mainresult}(\ref{reductivecase}) the $\delta-$Galois group of the $\delta-F-$module $A$ is a reductive linear algebraic group.
	\end{remark}

\section{Derivations stabilizing maximal subfields}\label{maxsubfield-stability}

 Let $A$ be a $\delta-F-$central simple algebra having a maximal subfield $E$  stabilized by $\delta_A,$ that is  $\delta_A(E)\subseteq E$. 
 When $F^{\delta}$ is an algebraically closed field, we show in Theorem \ref{derivation-maximalsubfields} of this section  that the identity component of the differential Galois group of the $\delta-F-$module $A$ is a torus. To accomplish this, we need the following facts from the theory of  liouvillian Picard-Vessiot extensions.
 
\subsection{Liouvillain field extensions} A $\delta-$field extension $K$ of $F$ is called a \emph{liouvillian extension} of $F$ if $K=F(t_1,\dots, t_n)$ such that for each $i=1,2,\dots,n,$ either \begin{enumerate}[(i)] \item \label{algebraic}$t_i$ is  algebraic over $F(t_1,\dots, t_{i-1})$ or \item\label{integral} $\delta_K(t_i)\in F(t_1,\dots, t_{i-1})$ (that is,  $t_i$ is an \emph{antiderivative} of an element of $F(t_1,\dots, t_{i-1})$) or \item\label{exponential} $t_i\neq 0$ and $\delta_K(t_i)/t_i\in F(t_1,\dots, t_{i-1})$ (that is, $t_i$ is an \emph{exponential of an antiderivative} of $F(t_1,\dots, t_{i-1})$). \end{enumerate} 
A liouvillian extension where the elements $t_1,\dots, t_n$ are of type (\ref{algebraic}) or (\ref{integral}) (respectively (\ref{algebraic}) or (\ref{exponential})) is called a \emph{primitive extension} (respectively, an \emph{exponential extension}).

%Let $K$ be a Picard-Vessiot extension of $F$ for some $\delta-F-$module $M.$ Then $K$ is a liouvillian extension of $F$ (respectively $K$ is a primitive extension of $F$, respectively $K$ is an exponential extension of $F$) if and only if the identity component of the  $\delta-$Galois group $\mathscr G$ of $M$ is a solvable algebraic group (respectively an unipotent algebraic group, respectively a torus). 

Let $K$ be a Picard-Vessiot extension of $F$ for some $\delta-F-$module $M.$ If the identity component of the  $\delta-$Galois group $\mathscr G$ of $M$ is a solvable algebraic group, then we have 
$K=F^0(\xi_1,\dots, \xi_l)(\eta_1,\dots, \eta_n),$ where  $F^0$ is the algebraic  closure of $F$ in $K$ and for each $i,$ $1\leq i \leq l,$  $\delta_K(\xi_i)/\xi _i\in F^0,$  $\delta_K(\eta_1)\in F^0(\xi_1,\dots, \xi_l)$ and for each $i,$ $2\leq i\leq n,$ $\delta_K(\eta_i)\in F^0(\xi_1,\dots, \xi_l)(\eta_1,\dots, \eta_{i-1}).$  If $\mathscr U$ denotes the unipotent radical of $\mathscr{G}$ then the $\delta-$field fixed by $\mathscr U$ $$K^\mathscr U:= \{x\in K\ | \ \sigma(x)=x, \ \text{for all} \ \sigma\in \mathscr U\}$$  is $F^0(\xi_1,\dots, \xi_l).$  
In fact, $K^\mathscr U=F^0(\xi_1,\dots, \xi_l)$ is a Picard-Vessiot extension of a $\delta-F-$module $N\in \{\{M\}\}$ and the $\delta-$Galois group of $N$ is isomorphic to the torus $\mathscr{G}/\mathscr{U}.$ Thus, in particular, if the identity component of $\mathscr G$  is a torus, then $K= F^0(\xi_1,\dots, \xi_l)$ where each $\delta_K(\xi_i)/\xi_i\in F^0$  and if the identity component of $\mathscr G$ is a unipotent algebraic group then $K=F^0(\eta_1,\dots,\eta_n),$ where $\delta_K(\eta_1)\in F^0$ and for $2\leq i\leq n,$ $\delta_K(\eta_i)\in F^0(\eta_1,\dots, \eta_{i-1})$ (for details see \cite[Proposition 6.7, Theorem 6.8]{Mag-94}).  If $K$ is contained in an exponential extension (respectively, a primitive extension) of $F$ then $K$ itself is an exponential extension (respectively, a primitive extension) of $F$ and $\mathscr G^0$ is a torus (respectively, unipotent algebraic group).

\begin{theorem}\label{derivation-maximalsubfields}
	Let $A$ be a central simple $F-$algebra having a maximal subfield $E.$ Let $\delta_A$ be a derivation on $A$ such that $F^{\delta}$ is algebraically closed and  $\delta_A(E)\subseteq E.$ Then  the identity component of the $\delta-$Galois group of the $\delta-F-$module $A$ is a torus or equivalently, the Picard-Vessiot extension $K$ of $F$ for the $\delta-F-$module $A$ is an exponential extension: $K=F^0(\xi_1,\dots, \xi_m),$ where $F^0$ is the algebraic closure of $F$ in $K$ and for each $i=1,\dots, m,$ $\xi_i\neq 0$ and $\delta_K(\xi_i)/\xi_i\in F^0.$
	\end{theorem}

\begin{proof}
It is well known that  maximal subfields of a central simple algebra are splitting fields. Therefore, $A\otimes E\cong \mathrm{M}_n(E)$ as $E-$algebras.  Let $L$ be a finite Galois extension of $F$ containing $E.$ Then, we have an isomorphism $\phi: A\otimes L\rightarrow \mathrm{M}_n(L)$ of $L-$algebras. Since $F$ is of characteristic zero, we shall choose $\alpha\in E$ such that $E=F(\alpha).$ If $G(L|F)=\{\sigma_1,\dots, \sigma_n\}$ is the (ordinary) Galois group of $L$ over $F$ then we have the following isomorphism of $L-$algebras: $$E\otimes L\cong L^n;\quad \alpha\otimes l\mapsto (\sigma_1(\alpha)l, \dots, \sigma_n(\alpha)l),\quad 1\otimes l\mapsto l.$$  
 Thus $\phi(E\otimes L)$ is isomorphic to the $L-$algebra $L^n,$ which in turn is isomorphic to the $L-$algebra $\mathfrak{D}$ of all diagonal matrices in $\mathrm{M}_n(L).$  We may therefore assume $\phi$ maps $E\otimes L$ onto $\mathfrak{D}$ (\cite[Lemma 2.2.9]{GSz}).  Transporting  the derivation $\delta_A \otimes \delta_L$ on $A\otimes L$ through $\phi$ to $\mathrm{M}_n(L)$, we have a derivation $D$ on $\mathrm{M}_n(L)$ that extends the derivation on $L$ and also stabilizes $\mathfrak{D}.$ Since $\delta^c$ also extends the derivation on $L,$  there is a matrix $P=(p_{ij})\in \mathrm{M}_n(L)$ such that $$D=\delta^c+\mathrm{inn}_P.$$
 We shall now show that $P$ must be a diagonal matrix.   For any $i, 1\leq i\leq n,$ $D(\mathtt{E}_{ii})=\mathtt{E}_{ii}P-P\mathtt{E}_{ii}=\sum_k \mathtt{E}_{ik}p_{ik}-\sum_l \mathtt{E}_{li}p_{li}$ and since $D(\mathtt{E}_{ii})\in \mathfrak D,$ it follows that $p_{ik}=0$ whenever $k\neq i$. Thus $P$ is diagonal. 
 
 Let $M$ be an $n-$dimensional $\delta-L-$module  with  $\delta_M(e_i)=-\sum^n_{j=1}p_{ji}e_i,$ where $e_1,\dots, e_n$ is an $L-$basis of $M.$ Let  $\mathcal L$ be a Picard-Vessiot extension of $L$ for $M$ and $(f_{ij})\in \mathrm{GL}_n(\mathcal L)$ be a fundamental matrix with  $\delta^c(f_{ij})=P(f_{ij}).$ Then for each $i=1,\dots, n,$ there is a nonzero entry $f_i$ in the $i-$th row of $(f_{ij})$ such that $\delta_\mathcal L(f_i)=p_{ii}f_i.$ Moreover, any two entries of the same row of $(f_{ij})$ are constant multiplies of each other. Thus, $\mathcal L=L(f_1,\dots ,f_n)$ is an exponential extension of $L.$ 
 
 It can be seen that the conjugation map $X\mapsto (f_{ij})X(f_{ij})^{-1}$ is a $\delta-\mathcal L-$algebra isomorphism between $(\mathrm{M}_n(\mathcal L), \delta^c)$ and $(\mathrm{M}_n(\mathcal L), D)$  and thus $\mathcal L$ splits the $\delta-F-$algebra $A.$  Note that  $\mathcal L$ has $F^{\delta}$ as its field of constants as $$\mathcal L^{\delta}=L^{\delta}=F^{\delta}.$$ Now, since the exponential extension $\mathcal L$ of $F$ splits the $\delta-F-$algebra $A,$ by Theorem \ref{Mainresult}(\ref{existence-minimal}),
 the  Picard Vessiot extension $K$ of $F$ for the $\delta-F-$module $A$ can be embedded in the $\delta-$field $\mathcal L.$ Therefore, as noted earlier, $K$ must be an exponential extension of $F$ and the identity component of the $\delta-$Galois group of $A$ must be a torus.\end{proof}
 
\begin{remarks*}
	
\item	\begin{enumerate}
		\item  From the proof of the Theorem \ref{derivation-maximalsubfields}, we observe that $\mathcal L:=F(f_1,\dots, f_n)$ splits the $\delta-F-$algebra $A.$ Therefore, it follows that $$\mathrm{deg}^\delta_{sp}(A)\leq n=\sqrt{\mathrm{dim}_F(A)}.$$  This generalizes the existing results on finding bounds for $\mathrm{deg}^\delta_{sp}(A),$ when $A$ is a differential quaternion algebra or a differential symbol algebra  having a differential maximal subfield (see \cite{AKVRS, AKG}) to an arbitrary $\delta-F-$central simple algebra $A$ having a differential maximal subfield.
	
	 \item  It is worthy to observe that any central division $F-$algebra $A$ has a maximal subfield $K$ and that every derivation on $F$ can be extended to a derivation $\delta_A$ on $A$ such that $\delta_A(K)\subseteq K.$ Thus, if $F$ is a $\delta-$field with $F^{\delta}$ algebraically closed\footnote{For example, take $F=\mathbb C(t,e^t),$ the rational function field in two variables over the complex numbers, with the derivation $d/dt.$} then  every central division $F-$algebra $A$ has a derivation $\delta_A$ such that the identity component of the $\delta-$Galois group of the $\delta-F-$module $A$ is solvable.
\end{enumerate}	
\end{remarks*}

\subsection{Groups with trivial torsor.} \label{groupswithtrivialtorsor} In Theorem \ref{Mainresult} (\ref{liouvilliancase}), we saw that for a $\delta-F-$central simple algebra $A$ having a $\delta-$Galois group whose identity component is solvable, the algebraic closure $F^0$  of $F$ in the Picard-Vessiot extension of the $\delta-F-$module $A$ splits the $F-$algebra $A.$   As we shall see in Proposition \ref{criteria-torsor}, this phenomenon of  $F^0$ splitting the central simple algebra $A$  is in fact true for any  $\delta-F-$central simple algebra that has a $\delta-$Galois group  whose identity component has only  trivial torsors.  Theorem \ref{Mainresult}(\ref{trivialtorsor}) is an easy consequence of Proposition \ref{criteria-torsor}.

Let $K$ be a Picard-Vessiot extension of $F$ for  a  $\delta-F-$module $M,$ $\mathscr R$ be its Picard-Vessiot ring and $\mathscr G$ be the $\delta-$Galois group of $M$ then $\mathrm{maxspec}(\mathscr R)$  is a $\mathscr G-$torsor over $F.$  If this torsor happens to be trivial, that is, there is an $F-$algebra homomorphism from $\mathscr R\to F$ then there is a $\mathscr G-$equivariant $F-$isomorphism between the Picard-Vessiot extension $E$ and the field of fractions of $F\otimes_{F^{\delta}} F^{\delta}[\mathscr G],$ where $F^{\delta}[\mathscr G]$ is the coordinate ring of the algebraic group $\mathscr{G}.$  Thus, in particular, if $\mathscr G$ is connected, then since $\mathscr G$ is a rational variety, $E$ becomes a purely transcendental extension of $F$ (see \cite[Theorem 5.19]{Mag-94}).

\begin{proposition}\label{criteria-torsor}
	Let $A$ be a $\delta-F-$central simple algebra, $F^{\delta}$ be an algebraically closed field, $K$ be a Picard-Vessiot extension of $F$ for the $\delta-F-$module $A$, $F^0$ be the algebraic closure of $F$ in $K$ and $\mathscr G$ be the $\delta-$Galois group of the $\delta-F-$module $M.$ If the Picard-Vessiot ring $\mathscr R$ of $K$ is a trivial $\mathscr G^0-$torsor over $F^0$ then $A\otimes F^0$ is a split $F^0-$algebra.
\end{proposition}

\begin{proof}
	The $\delta-F^0-$module $(A\otimes F^0, \delta_A\otimes \delta_{F^0})$ is split by the Picard-Vessiot extension $E$ and the $\delta-$Galois group $\mathscr{G}(E|F^0)=\mathscr G^0$ is a connected group. If the Picard-Vessiot ring  $\mathscr R$ of $K$ is a trivial $\mathscr G^0-$torsor over $F^0$ then as noted earlier, $E$ is a purely transcendental extension of $F^0$ which also splits the algebra $A\otimes F^0.$ Then, the Severi-Brauer variety $\mathrm{SB}(A\otimes F^0)$ associated to $A\otimes F^0$ has  a point in the purely transcendental extension $E$ of $F^0.$ This implies, $\mathrm{SB}(A\otimes F^0)$ must already be having a point in $F^0$ and thus $A\otimes F^0$ must be a split $F^0-$algebra (see \cite[Theorems 3.8.6, 3.8.7]{Jac-96}).
\end{proof}

\begin{remark}
Let $A$ be a $\delta-F-$central simple algebra with $F^{\delta}$ an algebraically closed field. Since the algebraic groups $\mathrm{GL}_n(F^{\delta}),$ $\mathrm{SL}_n(F^{\delta})$ and  connected solvable groups over $F^{\delta}$ have only trivial torsors, the $\delta-$Galois groups of the $\delta-F-$module $A$ cannot be isomorphic to any of these connected groups unless $A$ is a split $F-$algebra.
\end{remark}

%%%%%%%%%%%%%%%%%%%%%%%%%%%%%%%%%%%%%%%%%%%%%%%%%%%%%%%%%%%%%%%%%%%%%%%%%%%%%%%%%%%%%%%%%%%%%%%%%%%%%%%%%%%%%%%%%%%%%%%%%%%%%%%%%%%%%%%%%%%%%%%%%%%%%%%%%%%%%%%%%%%%%%%%%%%%%%%%%%%%%%%%%%%%%%%%%%%%%%%%%%%%%%%%%%%%%%%%%%%%%%%%%%%%%%%%%%%%%%%%%%%%%%%%%
\bibliographystyle{alpha}
\bibliography{MMVRS}
	\end{document}